\RequirePackage{luatex85} 
\documentclass[11pt]{amsart}
\usepackage{amssymb}
\usepackage{latexsym}
\usepackage{graphicx}
\usepackage{subfigure}
\usepackage{tikz}
\usepackage[linktocpage=true]{hyperref}
\usepackage[left=3.8cm, right=3.8cm, top=3cm, bottom=3cm]{geometry}
\usepackage{color}
\usepackage{here}
\usepackage[backgroundcolor = lightgray,
            linecolor = lightgray,
            textsize = tiny]{todonotes}
\usepackage{stackengine}
\usepackage{mathrsfs}
\usepackage[all]{xy}
\usepackage{scrextend}
\usepackage{enumitem}

\hypersetup{ colorlinks   = true, 
urlcolor  = teal, 
linkcolor    = magenta, 
citecolor   = cyan 
}

\def\a{{\alpha}}
\def\e{{\varepsilon}}
\def\beq{\begin{equation}}
\def\eeq{\end{equation}}

\def\a{\alpha}

\def\e{\epsilon}

\newcommand{\cC}{{\mathcal C}}

\newcommand{\cF}{{\mathcal F}}

\newcommand{\cP}{{\mathcal P}}

\newcommand{\cW}{{\mathcal W}}
\newcommand{\cWcu}{{\mathcal W}^{cu}}
\newcommand{\cWcs}{{\mathcal W}^{cs}}
\newcommand{\cWc}{{\mathcal W}^{c}}
\newcommand{\cWu}{{\mathcal W}^{u}}
\newcommand{\cWs}{{\mathcal W}^{s}}

\newcommand{\cS}{{\mathcal S}}

\newcommand{\cU}{{\mathcal U}}

\newcommand{\cJ}{{\mathcal J}}

\newcommand{\wt}{\widetilde}

\newcommand{\Z}{{\mathbb Z}}
\newcommand{\R}{{\mathbb R}}

\newcommand{\D}{{\mathbb D}}
\newcommand{\T}{{\mathbb T}}

\newcommand{\N}{{\mathbb N}}

\newcommand{\fc}{{\wt{f}_c}}
\newcommand{\xc}{{[\wt{x}]}}
\newcommand{\yc}{{[\wt{y}]}}
\newcommand{\zc}{{[\wt{z}]}}

\newtheorem{thm}{Theorem}[section]

\newtheorem{lem}[thm]{Lemma}

\newtheorem{prop}[thm]{Proposition}
\newtheorem{cor}[thm]{Corollary}
\newtheorem{claim}[thm]{Claim}

\theoremstyle{definition}
\newtheorem{df}[thm]{Definition}
\newtheorem{rem}[thm]{Remark}

\theoremstyle{plain}
\newtheorem*{main thm}{Main Theorem}
\newtheorem*{main thm bis}{Main Theorem (alternate version)}

\newtheorem{theoalph}{Theorem}

\newcommand{\subscript}[2]{$#1 _ #2$}

\providecommand{\abs}[1]{\lvert#1\rvert}
\providecommand{\norm}[1]{\lVert#1\rVert}

\begin{document}
\numberwithin{equation}{section}

\title[Coherence of F.L.P.H.]{Dynamical coherence in isotopy classes of fibered lifted partially hyperbolic diffeomorphisms}
\author[Luis Pedro Pi\~{n}eyr\'ua]{Luis Pedro Pi\~{n}eyr\'ua$^{*}$}
\thanks{$^{*}$L.P.P was partially supported by CAP's doctoral scholarship and CSIC group 618}

\author[Mart\'in Sambarino]{Mart\'in Sambarino$^{**}$}
\thanks{$^{**}$M.S. was supported by CSIC group 618}


\maketitle

\begin{abstract}
	We introduce the notion of \textit{fibered lifted partially hyperbolic diffeomorphisms} and we prove that any partially hyperbolic diffeomorphism isotopic to a fibered lifted one where the isotopy take place inside partially hyperbolic systems is \textit{dynamically coherent.} Moreover we prove some global stability result: every two partially hyperbolic diffeomorphisms in the same connected component of a fibered lifted partially hyperbolic diffeomorphisms, are leaf conjugate. 
\end{abstract}

\tableofcontents

\section{Introduction}

One of the major problems in dynamical systems is to determine whether two systems are equivalent or not from some point of view. In the topological sense this is done by the conjugacy problem: we say that two diffeomorphisms $f:M \to M$ and $g:M\to M$ are \textit{conjugate} if there is a homeomorphism $h:M\to M$ such that $f\circ h=h\circ g$. In the hyperbolic setting this problem was achieved by Franks \cite{Fr}, \cite{Fr1} and A. Manning \cite{Man} with the well-known results about the classification of the globally hyperbolic (Anosov) diffeomorphisms on nilmanifolds: any Anosov diffeomorphism on a nilmanifold is conjugated to its linear part. 

A fundamental tool in the proofs of these results is the stable manifold theorem, i.e. the integrability of the stable and unstable bundles of a uniformly hyperbolic diffeomorphism. Since these sub-bundles are transversal, their corresponding integrated foliations fill the space at least locally. 

In the partially hyperbolic case, given $f:M \to M$ with a splitting of the form $TM=E^s_f\oplus E^c_f\oplus E^u_f$, it is well known that the bundles $E^s_f$ and $E^u_f$ integrate into unique $f$-invariant foliations $\cWs_f$ and $\cWu_f$ \cite{HPS}. However, the center bundle $E^c_f$ can have many different behaviors and one hopes to be able to integrate the center bundle too, although this is not always the case. This represents the first important difference between global and partial hyperbolicity.

We say that a partially hyperbolic diffeomorphism $f$ is \textit{dynamically coherent} if the bundles $E^{s}_f\oplus E^{c}_f$ and $E^c_f\oplus E^u_f$  integrates to invariant foliations (and hence, the center bundle $E^c_f$ does too). Otherwise we say that $f$ is \textit{dynamically incoherent}. The first example of dynamically incoherent partially hyperbolic diffeomorphism was built in \cite{Wi} (see also \cite{BuW0}) on a six dimensional nilmanifold with 4-dim center bundle. Later in \cite{RHRHU1}, the authors built an example on the torus $\T^3$ (with 1-dim center bundle). In the later example on the 3-torus, the lack of differentiability on the bundles breaks uniqueness for the center bundle (although there are curves tangent to $E^c$ by Peano's theorem). In the 6-dimensional manifold example, despite having $C^1$ bundles, the Frobenius condition fails and thus no integrability is possible on the center bundle. 

It is an open question whether dynamical coherence is a $C^1$-open condition among partially hyperbolic systems. A closely related property is \textit{plaque expansiveness} which implies this in the affirmative way \cite{HPS}. Let us mention here that to the best of our knowledge, every known example of a dynamically coherent diffeomorphism is plaque-expansive, and therefore $C^1$ stably dynamically coherent. On the other hand, in \cite{BGHP} the authors built $C^1$ open sets of dynamically incoherent examples in certain Seifert 3-manifolds.  

In addition to these $C^1$ local results, in recent years the attention has been paid in the (apparently) relation between integrability of the bundles and the isotopy class of the map. The first result in this direction is \cite{FPS} where the authors obtained dynamical coherence for entire isotopy classes (in the partially hyperbolic setting) of linear Anosov diffeomorphisms on $\T^d$. This is the first result where the integrability of the center-stable and center-unstable bundle is obtained for a whole isotopy class of maps (the nilmanifold case of this result is proven in \cite{Pi}). In the same direction, in \cite{BFFP2} it is proven that in certain Seifert 3-manifolds, every partially hyperbolic diffeomorphism isotopic to the identity is dynamically coherent, and in \cite{Mar} dynamical coherence is obtained in isotopy classes of discretized Anosov flows. On the other hand, in \cite{BFFP} the authors (following the construction in \cite{BGHP}) obtained entire isotopy classes of dynamically incoherent partially hyperbolic diffeomorphisms. 

All these results are somehow surprising, since on the one hand integrability is quite hard and technical to get, and on the other hand there is a lot of freedom to move inside isotopy classes (and there is no assumption on the behavior on center bundles despite domination). By the previous evidence, it seems that integrability or not of the center-stable and center-unstable bundles is a phenomenon that depends directly on the isotopy class of the diffeomorphism (a purely topological property). 

Our contribution in this article is to go towards this direction by showing that similar results to the one  in \cite{FPS}  hold in a much large class of maps that includes: linear partially hyperbolic automorphisms, direct products, skew products, and more general fiberings over Anosov diffeomorphisms. Surprisingly, similar methods apply to all this classes with a unified proof. 

Let us first give the following definition in order to state precisely the main theorems of this paper. Given a dynamically coherent partially hyperbolic diffeomorphism $f:M\to M$ we will say that $f$ is \textit{fibered lifted}, if it verifies the following two conditions (the precise definition is given in Definition \ref{dffph} but the following capture the essence of it):
\begin{itemize}
	\item the foliations $\wt{\cW}^{cs}_f$ and $\wt{\cW}^{u}_f$ have global product structure in the universal cover $\wt{M}$ and the same happens with $\wt{\cW}^{cu}_f$ and $\wt{\cW}^{s}_f$.
	\item the quotient space $\wt{M}/\wt{\cWc_{f}}$ is a (topological) manifold and the induced map in the quotient by center leaves $\wt{f_c}:\wt{M}/\wt{\cWc_{f}}\to \wt{M}/\wt{\cWc_{f}}$ is a  hyperbolic homeomorphism.
\end{itemize}

Our main theorem is the following:

\begin{theoalph}\label{a} Let $f:M\to M$ be a fibered lifted partially hyperbolic diffeomorphism. Let $g$ be a partially hyperbolic diffeomorphism isotopic to $f$ such that the isotopy is inside the set of partially hyperbolic diffeomorphisms (preserving the dimension of the bundles). Then $g$ is dynamically coherent.
\end{theoalph}

The natural concept of stability in the partially hyperbolic setting is, according to \cite{HPS}, leaf conjugacy.  We say that two dynamically coherent partially hyperbolic diffeomorphisms $f$ and $g$ are \textit{leaf conjugate} if there exists a homeomorphism $h \colon M\to M$, called a leaf conjugacy, such that  $h$ maps a $f$-center leaf to a $g$-center leaf, and $h \circ f(\mathcal{W}_f^c(\cdot))= g \circ h(\mathcal{W}_f^c(\cdot))$. In other words, the center foliations are the same as well as the dynamics of the center leaves up to a homeomorphism. We  prove the following global stability result. 

\begin{theoalph}\label{b}
	Let $f:M\to M$ be a fibered lifted partially hyperbolic diffeomorphism. Then every partially hyperbolic diffeomorphism $g$ which is isotopic to $f$ such that the isotopy is inside the set of partially hyperbolic diffeomorphisms (preserving the dimension of the bundles) is leaf conjugate to $f$.
\end{theoalph}

 Recently, Doucette \cite{D}, proved that fibered partially hyperbolic diffeomorphisms over nilmanifolds with compact fibers are leaf conjugate to a smooth model that is isometric on the fibers and descends to a hyperbolic nilmanifold automorphism on the base. We may ask if our fibered lifted partially hyperbolic diffeomorphisms are also leaf conjugate to a smooth models, for instance as the ones described in Section \ref{subsecexamples}.

\subsection*{Organization of the paper}

In Section \ref{sectionpreliminaries} we present some preliminaries and well known results. In Section \ref{sfph} we introduce in detail the fibered lifted partially hyperbolic diffeomorphisms, we present some examples and we obtain a stability result. We then pass to Section \ref{sectionintegrability} where we prove an integrability criterion for partially hyperbolic diffeomorphisms isotopic to fibered lifted partially hyperbolic diffeomorphisms, which we use in Section \ref{sectionopenandclosed} to obtain dynamically coherence in the whole isotopy class of a fibered lifted partially hyperbolic diffeomorphism and prove Theorem \ref{a}. Finally in Section \ref{sectionleafc} we deal with leaf conjugacy and prove Theorem \ref{b}.

\section*{Acknowledgments} 
The authors would like to thank Rafael Potrie for fruitful conversations and suggestions about the paper. 

\section{Preliminaries} \label{sectionpreliminaries}

We say that a diffeomorphism $f:M\to M$ is \textit{partially hyperbolic} if there exists a $Df$-invariant splitting $TM=E^{s}_f\oplus E^c_f\oplus E^{u}_f$ such that there exists $N>0$ and $\lambda>1$ verifying that for every $x\in M$ and unit vectors $v^{\sigma}\in E^{\sigma}_f$, with $\sigma=s,c,u$ we have:
\begin{itemize}
	\item $\lambda \norm{Df_x^N(v^{s})}< \norm{Df_x^N(v^c)}< \lambda^{-1}\norm{Df_x^N(v^{u})}$, and
	\item $\norm{Df_x^N(v^{s})}<\lambda^{-1}<\lambda < \norm{Df_x^N(v^{u})}$.
\end{itemize} 
We are going to assume that $N=1$ due to N. Gourmelon \cite{Gou} work on adapted metrics. We will note by $\text{PH}(M)$ to the set of partially hyperbolic diffeomorphisms acting on the manifold $M$. 

It is well known that if $f\in \text{PH}(M)$ then the stable and unstable bundles $E^s_f$ and $E^u_f$ are only Hölder continuous. Nevertheless the celebrated stable manifold theorem says that the strong bundles $E_f^u$ and $E_f^s$ are uniquely integrable \cite{HPS}. Their corresponding unique integral foliations are called the \textit{strong unstable} and \textit{strong stable} foliations respectively, and we note them by $\cWu_f$ and $\cWs_f$. Notice that since $E^u_f$ and $E^s_f$ are $Df$-invariant, then unique integrability (or having a unique integral foliation) implies that their corresponding integral foliations $\cW_f^{\sigma}$ are invariant under the dynamics, i.e.,  $f(\cW_f^{\sigma}(\cdot))=\cW_f^{\sigma} (f(\cdot))$ for $\sigma=s,u$. 

Despite the stable manifold theorem, we don't have \textit{a priori} integrability of the rest of the bundles $E^{cs}_f$, $E^{cu}_f$ and $E^c_f$. This fact leads to the following definition.

\begin{df} \label{dfdc}
	A partially hyperbolic diffeomorphism $f$ is \textit{dynamically coherent} if the center-unstable bundle $E_f^{cu}:=E_f^c \oplus E_f^u$ and the center-stable bundle $E_f^{cs}:=E_f^c \oplus E_f^s$ integrate to invariant foliations. They are called the center-unstable foliation, resp. the center-stable foliation and are denoted by $\cW^{cu}_f$, $\cW^{cs}_f$.
\end{df}

Notice that dynamical coherence implies that the center distribution $E^c_f$ integrates to an invariant foliation too: if $f\in \text{PH}(M)$ is dynamically coherent, then for any $x \in M$ the set $\cWc_f(x):= \cWcs_f(x) \cap \cWcu_f(x)$ integrates $E^c_f$ to an invariant foliation and we call $\cWc_f$ the center foliation. On the other hand, the integrability of $E^c_f$ does not imply dynamically coherence: if $E^c_f$ integrates into $\cWc_f$ and if we take $\cWcs_f(x)=\cup_{y\in \cWc_f(x)}\cWs_f(y)$ we obtain a plaque tangent to $E^{cs}_f(x)$ but the union of this plaques is not going to be a foliation necessary.

We want to remark also that we don't require unique integrability of the bundles $E^{cs}_f$ and $E^{cu}_f$ in the definition above, although every known example of dynamical coherence is uniquely integrable. In \cite{BuW0} there is a long discussion about all the possible definitions of dynamical coherence that have been used since its introduction and every implications between them.

It is an open question whether dynamical coherence is a $C^1$-open condition among $\text{PH}(M)$. A closely related property is \textit{plaque expansiveness}. Before introducing it, we need a few definitions. We denote by $d$ the distance in $M$ induced by the riemannian metric. 

\begin{df} \label{defpseor}
	Given $\e>0$ we say that a sequence of points $\{x_n\}_{n\in \Z} \subset M$ is a $\e$-\textit{pseudo orbit} of $f$ if $d(f(x_n),x_{n+1})<\e$ for every $n\in \Z$. In addition, if $f$ preserves a foliation with plaquation  $\cP$, we say that the pseudo orbit respects $\cP$ if $f(x_n)$ and $x_{n+1}$ lie in a common plaque of $\cP$ for every $n \in \Z$. \end{df} 

\begin{df}\label{plqauexp}
	We say that $f\in \text{PH}(M)$ is \textit{plaque expansive} (see \textnormal{\cite[Section 7]{HPS}}) if there exists exists $\varepsilon>0$ with the following property: if $(p_n)_{n \geq 0}$ and $(q_n)_{n \geq 0}$ are $\varepsilon$-pseudo orbits which respect $\cP$ and such that $d(p_n,q_n)\leq \varepsilon$ for all $n \geq 0$, then $q_n$ and $p_n$ lie in a common plaque of $\cP$. 
\end{df}

It is known that plaque expansiveness is a $C^1$-open condition (see Theorem 7.4 in \cite{HPS}). The importance of plaque expansivity lies on the following theorem.

\begin{thm}[Theorem 7.1 \cite{HPS}, see also Theorem 1 in \cite{PSW1}]\label{thmplaqueexpansivetostablydc0}
	Let us assume that $f$ is dynamically coherent and plaque expansive. Then any $g \in \textnormal{PH}(M)$ which is sufficiently $C^1$-close to $f$ is also dynamically coherent and plaque expansive. Moreover, there exists a homeomorphism $h=h_g \colon M\to M$,
		such that  $h$ maps a $f$-center leaf to a $g$-center leaf, and $h \circ f(\mathcal{W}_f^c(\cdot))= g \circ h(\mathcal{W}_f^c(\cdot))$. 
\end{thm}

As a result, every $f\in \text{PH}(M)$ dynamically coherent and plaque-expansive is $C^1$ stably dynamically coherent. The problem then, is to decide when a partially hyperbolic diffeomorphism is plaque exapansive. This question is open in its full generality although plaque expansivity has been obtained in several cases:

\begin{itemize}
	\item when the center foliation $\cWc_f$ is $C^1$ (or $E^c_f$ is $C^1$ or both $E^{cs}_f$ and $E^{cu}_f$ are $C^1$) this was proved in \cite{HPS}.
	\item when $Df|_{E^c_f}$ is an isometry this was proved in \cite{RHRHU0}, originally mentioned in \cite{HPS} without proof.
	\item when the center foliation $\cWc_f$ is uniformly compact, i.e. every center leaf is compact and there is a uniform bound on the volumes, proved in \cite{BB}, \cite{Car}. 
	\item when $f$ is a discretized Anosov flow \cite{Mar}.
\end{itemize}

We finish this section by recalling another definition (which we already mention in the introduction) that arises from Theorem \ref{thmplaqueexpansivetostablydc0} and it is related with the topological stability of a partially hyperbolic diffeomorphism.

\begin{df} We say that two dynamically coherent partially hyperbolic diffeomorphisms $f,g:M \to M$ are \textit{leaf conjugate} if there exists a homeomorphism $h\colon M\to M$, called a \textit{leaf conjugacy}, such that  $h$ maps a $f$-center leaf to a $g$-center leaf, and $h \circ f(\mathcal{W}_f^c(\cdot))= g \circ h(\mathcal{W}_f^c(\cdot))$. 
\end{df}

Leaf conjugacy is the analogous to conjugacy for Anosov diffeomorphisms in the partially hyperbolic case (notice that we need dynamically coherence for this definition to make sense). Hence by Theorem \ref{thmplaqueexpansivetostablydc0} every $f\in \text{PH}(M)$ dynamically coherent and plaque expansive is topologically stable in the sense mentioned above.

\section{Fibered lifted partially hyperbolic diffeomorphisms} \label{sfph}

\subsection{Definitions and quotient dynamics}
Let $(X,dist)$ a metric space. Given a subset $K \subset X$ and $R>0$ we call $B(K,R)$ the $R$-neighbourhood of $K$, that is, the set of points in $X$ that are less than $R$ from some point in $K$:
$$B(K,R)=\{ x\in X: \text{there is} \ y\in K \ \text{s.t.} \  dist(x,y)<R\}
$$ This includes the case $K=\{x\}$ and $B(x,R)=\{y\in X: dist(x,y)<R\}
$. 

Given a homeomorphism $f:X\to X$ on a metric space $(X,dist)$, we define the \textit{stable set} and  the \textit{unstable set} of a point $x\in X$ as the sets:  
\begin{align*}
	\cWs_f(x)&=\{y \in X: dist(f^{n}(x),f^{n}(y))\to_{n \to +\infty} 0\} \\ \cWu_f(x)&=\{y \in X: dist(f^{-n}(x),f^{-n}(y))\to_{n \to +\infty} 0\} \end{align*}
	
For $r>0$ we define the \textit{stable and unstale disk of size $r$} the following:	
	\begin{align*}
	D^s_f(x,r)&=\{y \in \cW^s_f(x): dist(x,y)<r\} \\ D^u_f(x,r)&=\{y \in \cW^u_f(x): dist(x,y)<r\}\end{align*}

For the next definition, one may  have in mind  the case of a linear hyperbolic map on a euclidean space.

\begin{df}\label{dfhyphom}
	We say that a homeomorphism $f:X \to X$ on a (non-compact) metric space $(X,dist)$ is \textit{globally uniformly hyperbolic} if the following holds: 
	\begin{enumerate}
	
		\item\label{c1hyphom} The stable and unstable sets have a uniform behaviour: 
		 there exists $n\in \N$ such that for every $r>0$ we have:
		
		\begin{itemize}
		\item  $f^n(D^s_f(x,r))\subset D^s_f(f^n(x),r/2).$
		\item $f^{-n}(D^u_f(x,r))\subset D^u_f(f^{-n}(x),r/2).$		\end{itemize}
	\item  \label{c2hyphom}\begin{itemize}
	\item If $x_1,x_2\in \cWs_{f}(x)$ then $dist(f^{-n}(x_1),f^{-n}(x_2))\to_{n\to\infty}\infty.$
	\item If $x_1,x_2\in \cWu_{f}(x)$ then $dist(f^{n}(x_1),f^{n}(x_2))\to_{n\to\infty}\infty.$
	\end{itemize}	
		\item \label{c3hyphom} There exists global product structure (GPS): for any $x, y\in X,$ then $\cWs_f(x)$ and $\cWu_f(y)$ intersect at exactly one point denoted by $[x,y]$ and this point depends continuously on $(x,y).$ 
	\end{enumerate}
\end{df}

According to Definition \ref{dfdc}, a dynamically coherent partially hyperbolic diffeomorphism $f:M\to M$ has two invariant foliations $\cW^{cu}_f$ and $\cW^{cs}_f$ tangent to $E_f^c \oplus E_f^{u}$ and $E_f^{s} \oplus E^c_f$ respectively. This implies in addition, that we have a center foliation $\cWc_f(x):= \cWcs_f(x) \cap \cWcu_f(x)$ which is also $f$-invariant and tangent to $E^c_f$. 

This center foliation $\cWc_f$ gives a partition of the manifold $M$ and thus we have a well defined quotient space $M/\cWc_f$.  Nevertheless, unless the center foliation is quite particular, this quotient space is wild. It is therefore convenient to lift first to the universal cover and then to consider the quotient space. As we will see, in many examples this is very well behaved.

Let $\pi:\wt{M}\to M$ be the universal cover of $M$ and recall that $M=\wt{M}/\Gamma$ where $\Gamma=\pi_1(M)$ acts on $\wt{M}$ by isometries. In this case we have that $\wt{\cW}^c_f$ gives a partition of the manifold $\wt{M}$ and we have a well defined quotient space $\wt{M}/\wt{\cW}^c_f$. We are going to note by $\wt{p}:\wt{M}\to \wt{M}/\wt{\cW}^c_f$ to the projection into equivalence classes, and its corresponding induced map will be:  $$\wt{f_c}:\wt{M}/\wt{\cW}^c_f\to \wt{M}/\wt{\cW}^c_f \ \ \ \text{given by} \ \ \ \wt{p}\circ \wt{f}=\wt{f_c}\circ \wt{p}$$ 
The idea of these maps is to ``cancel'' the non-hyperbolic behavior of the partially hyperbolic diffeomorphism $f$ in order to get some hyperbolicity in the quotient space. The following is the main object of this article.

\begin{df} \label{dffph}
	Let $f:M\to M$ be a dynamically coherent partially hyperbolic diffeomorphism of class $C^r$. We say that $f$ is \textit{fibered lifted} if:
	\begin{enumerate}[label=(\subscript{H}{{\arabic*}})]
		\item the foliations $\wt{\cW}^{cs}_f$ and $\wt{\cW}^{u}_f$ have global product structure; \\ the foliations $\wt{\cW}^{cu}_f$ and $\wt{\cW}^{s}_f$ have global product structure. \label{fph1}
		
		\item For every $\wt{x},\wt{y}\in \wt{M}$ we have that the Hausdorff distance $d_H(\wt{\cW}^c_f(\wt{x}),\wt{\cW}^c_f(\wt{y}))<\infty$.  \label{fph2}  In particular, it induces a distance \textit{dist} in the quotient space $\wt{M}_c:=\wt{M}/\wt{\cWc_{f}}$ 
		by\footnote{In general the Hausdorff distance is defined between compact sets, but it works in the same way for closed sets under the assumtpion that is finite and defined as $$d_H(\wt{\cW}^c_f(\wt{x}),\wt{\cW}^c_f(\wt{y})):=\max \left\{ \sup_{z\in \wt{\cW}^c_f(\wt{x})} d\left(z,\wt{\cW}^c_f(\wt{y})\right), \sup_{z\in \wt{\cW}^c_f(\wt{y})} d\left(z,\wt{\cW}^c_f(\wt{x})\right) \right\}$$}
		$$dist(\wt{p}(\wt{x}),\wt{p}(\wt{y})):=d_H(\wt{\cW}^c_f(\wt{x}),\wt{\cW}^c_f(\wt{y})).$$ 
		This distance is compatible with the quotient topology.
		
		\item the map $\wt{f_c}:\wt{M}/\wt{\cWc_{f}}\to \wt{M}/\wt{\cWc_{f}}$ is a  globally uniformly hyperbolic homeomorphism with
		\label{fph3}
		$$\cWs_{\wt{f}_c}([\wt{x}]):=\wt{p}(\wt{\cW}^{s}_f(\wt{x}))\;\;\mbox{ and }\;\;\;\cWu_{\wt{f}_c}([\wt{x}]):=\wt{p}(\wt{\cW}^{u}_f(\wt{x}))	$$
		
	\end{enumerate}
\end{df}
\begin{rem}\label{r.pi}
By the global product structure \ref{fph1} we also have that \textit{inside} center-unstable leaves, global product structure holds between center and unstable leaves: for any $\wt{y}, \wt{z}\in\wt{\cW}^{cu}_f(\wt{x})$ we have $\sharp(\wt{\cW}^c_f(\wt{y})\cap \wt{\cW}^{u}_f(\wt{z}))=1.$ The same for center-stable leaves. 
\end{rem}

\subsection{Examples} \label{subsecexamples}

The following are a few examples of fibered lifted partially hyperbolic diffeomorphisms. The list is not intended to exhaustive. 

\subsubsection*{Anosov automorphisms.}
Let $A\in \text{SL}(d,\Z)$ be a hyperbolic matrix with a splitting of the form $\R^d=E^{ss}_A\oplus E^{ws}_A\oplus E^{wu}_A\oplus E^{uu}_A$. This matrix induces an Anosov diffeomorphism $f:\T^d\to\T^d$. We can see $f$ as a fibered lifted partially hyperbolic with trivial fibers. In this case, $\R^d/\wt{\cW}^c_f=\R^d$ and $\wt{f_c}=A$ and the points in Definition \ref{dffph} are trivially satisfied.

On the other hand we can see $f$ as a partially hyperbolic diffeomorphism by taking the center bundle as $E^c_f=E^{ws}_A\oplus E^{wu}_A$. Since $E^c_A$ is a linear subspace, we get that $f$ is dynamically coherent and moreover $f$ has global product structure (as in \ref{fph1}). The quotient space is $\R^d/\wt{\cW}^c_f=E^{ss}_A \oplus E^{uu}_A=\wt{M}_c$ and the map $\wt{p}$ can be seen as the orthogonal projection $\Pi^{su}:\R^d\to E^{ss}_A\oplus E^{uu}_A=\wt{M}_c$ proving point
\ref{fph2}. The quotient map is $\wt{f_c}=A|_{E^{ss}_A\oplus E^{uu}_A}$ and we get \ref{fph3}. Therefore $f$ is a fibered lifted partially hyperbolic diffeomorphism.

\subsubsection*{Partially hyperbolic automorphisms.}
Let $A \in \text{SL}(d,\Z)$ be a matrix with a splitting of the form $\R^d=E^{s}_A\oplus E^c_A \oplus E^{u}_A$, where $E^c_A$ is the generalized eigenspace associated to the eigenvalues of modulus equal to one. Like in the Anosov case above, the matrix $A$ induces a map $f:\T^d\to \T^d$ which is a dynamically coherent partially hyperbolic diffeomorphism. Since $f$ is linear, it's clear that $f$ has global product structure as in \ref{fph1}, the quotient space is $\R^d/\wt{\cW}^c_f=E^{s}_A\oplus E^{u}_A$ and the map $\wt{p}$ is the orthogonal projection $\Pi^{su}:\R^d\to E^{s}_A\oplus E^{u}_A$ proving \ref{fph2}. Finally observe that $\wt{f_c}=A|_{E^{s}_A\oplus E^{u}_A}$ and thus we get point \ref{fph3}. 

\subsubsection*{Automorphisms on nilmanifolds.} The same examples mentioned above can be carried out in the nilmanifold case. Let $G$ be a connected, simply connected nilpotent Lie group and denote by $\mathfrak{g}$ its Lie algebra. Let $A:G \to G$ be a Lie group isomorphism, and let $\Gamma$ be a discrete and cocompact subgroup which is $A$-invariant. Then $A$ induces a map $f_A$ in the corresponding quotient space $M=G/\Gamma$, given by $f_{A}(x \cdot \Gamma)=A(x) \cdot \Gamma$. The differential $DA_{e}:T_{e}G \to T_{e}G$ is a linear isomorphism and induces a Lie algebra isomorphism between the corresponding Lie algebras $dA:\mathfrak{g} \to \mathfrak{g}$. This correspondence between $\mathfrak{g}$ and $T_{e}G$ comes from the linear isomorphism $\alpha:\mathfrak{g}\to T_{e}G$ which sends $X\in \mathfrak{g}$ to the vector $X(e)\in T_{e}G$ and it also conjugates the maps $DA_{e}$ and $dA$: $DA_{e}\circ \alpha = \alpha \circ dA$. We say that $f_A$ is \textit{partially hyperbolic} if $dA: \mathfrak{g} \to \mathfrak{g}$ admits a partially hyperbolic splitting of the form: 
$$\mathfrak{g}=\mathfrak{g}^{s}\oplus \mathfrak{g}^{c}\oplus \mathfrak{g}^{u}$$
where $\mathfrak{g}^{s}$, $\mathfrak{g}^{c}$ and $\mathfrak{g}^{u}$ are called the \textit{strong stable}, \textit{center} and \textit{strong unstable} subspaces respectively. The direct sums $\mathfrak{g}^{cs}=\mathfrak{g}^{s}\oplus \mathfrak{g}^{c}$ and $\mathfrak{g}^{cu}=\mathfrak{g}^{c}\oplus \mathfrak{g}^{u}$ are the \textit{central stable} and \textit{central unstable} subspaces. It's easy to see that $\mathfrak{g}^{s}$ and $\mathfrak{g}^{u}$ are Lie subalgebras, however, we are going to assume that every subspace mentioned above is a Lie subalgebra of $\mathfrak{g}$. Now we can choose an appropriate inner product in $\mathfrak{g}$ and send it to the tangent space $T_{e}G$ by the isomorphism $\alpha :\mathfrak{g} \to T_{e}G$, and translating by left multiplication we get a Riemannian metric which is adapted to this splitting. In the same way we define the distributions $E^{*}_{f_A}(x):=L_x(\alpha(\mathfrak{g}^{*}))$ for $*=s,cs,c,cu,u$. Since the group $G$ is nilpotent, every Lie subalgebra $\mathfrak{g}^{*}$ is integrable and moreover, their corresponding integral subgroup $G^{*}$ are just the image by the exponential map, i.e.  $\exp(\mathfrak{g}^{*})=G^{*}$. Then the distributions $E^{*}_{f_A}(x)$ are integrable and their corresponding tangent foliations are given by $\wt{\cW}^{*}_{f_A}(x)=L_x(G^{*})$ for $*=s,cs,c,cu,u$. It can be seen that the foliations $\wt{\cW}^{cs}_{f_A}$ and $\wt{\cW}^{u}_{f_A}$ have GPS, and the same happens with  $\wt{\cW}^{cu}_{f_A}$ and $\wt{\cW}^{s}_{f_A}$ getting Point \ref{fph1}. Since the center foliation  $\wt{\cW}^{c}_{f_A}$ is defined by left multiplication (which are isometries) we get Point \ref{fph2}. Finally Point \ref{fph3} can be deduced from Proposition \ref{p.suficient} (see \cite{Pi} for more details). 

\subsubsection*{Anosov $\times$ Identity.} 
Let $A\in \text{SL}(d,\Z)$ be a hyperbolic matrix with a splitting of the form $\R^d=E^s_A\oplus E^u_A$ and let $f:\T^d\to \T^d$ be the induced Anosov diffeomorphism as above. Let $N$ be any other manifold of any dimension and let $g:\T^d\times N\to T^d \times N$ be the map $g=f\times \text{Id}$. Then $g$ is a dynamically coherent partially hyperbolic diffeomorphism with global product structure proving point \ref{fph1}. The center leaves are of the form $\cW^c_g(x,y)=\{x\}\times N$, and thus its quotient space is $(\R^d\times \wt{N})/\wt{\cW}^c_g=\R^d=E^{s}_A\oplus E^{u}_A$, the projection $\wt{p}_g:\R^d\times \wt{N}\to (\R^d \times \wt{N})/\wt{\cW}^c_g=\R^d$ is just the projection on the first coordinate and the induced map is $\wt{g_c}=A|_{E^s_A\oplus E^u_A}$. This shows points \ref{fph2} and \ref{fph3} proving that $g$ is fibered lifted partially hyperbolic.

\subsubsection*{Dominated splitting examples.} 
Generalizing the previous example take $f:M\to M$ any of the previous fibered lifted partially hyperbolic diffeomorphisms and let $N$ be a manifold of any dimension. Take a map $g:N\to N$ such that its behavior is dominated by $f$: there exist $\lambda\in (0,1)$ such that $\norm{Dg_y}\leq \lambda m(Df_x|_{E^u_f})$ and $\norm{Df_x|_{E^s_f}}\leq \lambda m(Dg_y)$ for every $x\in M$, $y\in N$. Then the map $F:M \times N \to M \times N$ defined by $F=f \times g$ is a dynamically coherent partially hyperbolic diffeomorphism. Since the center leaves are $\cW^c_F(x,y)=\cW^c_f(x)\times N$ the quotient space is $(\wt{M}\times \wt{N})/\wt{\cW}^c_F=\wt{M}/\wt{\cW}^c_f$. Moreover since $f$ is fibered lifted, we have that $F$ has global product structure as in \ref{fph1}. The projection $\wt{p}_F$ is the function $\wt{p}_F(\wt{x},\wt{y})=\wt{p}_f(\wt{x})$ where $\wt{p}_f:\wt{M}\to \wt{M}/\wt{\cW}^c_f$ proving \ref{fph2} and the induced map is just $\wt{F_c}=\wt{f_c}$ getting point \ref{fph3}.

\subsubsection*{Skew-products} Let $f:\T^d\to \T^d$ be a partially hyperbolic automorphism induced by some partially hyperbolic matrix as above (with the nilmanifold example works as well) and let $G$ be a compact Lie group. Take a smooth function $\theta:N \to G$ and consider the map $F:N\times G \to N\times G$ given by $F(x,g)=(f(x),\theta(x)g)$. Then it is easy to see that $F$ is a dynamically coherent partially hyperbolic diffeomorphism with global product structure in the universal cover proving \ref{fph1}. The center leaves are given by $\cW^c_F(x,g)=\{x\}\times G$, and therefore we have point \ref{fph2}. By the same reason the projection into equivalence classes $\wt{p}_F$ is just the projection into the first coordinate and the induced map is just $\wt{F_c}=\wt{f}$ getting point \ref{fph3} and therefore $F$ is a fibered lifted partially hyperbolic diffeomorphism.  Perturbations are not needed to be in the skew product setting.

\subsubsection*{Fiberings} More general than the previous examples, we have the systems that fiber over partially hyperbolic diffeomorphisms. Take $f:\T^d \to \T^d$ be a fibered lifted partially hyperbolic diffeomorphism with a splitting of the form $T\T^d=E^s_f \oplus E^c_f \oplus E^u_f$ (with the nilmanifold example mentioned above works as well). Take a fibration $N \hookrightarrow M \xrightarrow{\pi} \T^d$, i.e. $\pi^{-1}(\{x\}) \simeq N$ for every $x\in \T^d$, and denote by $N(x)=\pi^{-1}(\{x\})$ to the fiber through $x$. Consider a lift $F:M \to M$, i.e. a map such that $\pi \circ F=f \circ \pi$. Then if we ask for the lift $F$ to verify:
$$ \norm{Df_{\pi(x)}|_{E^s_f}}< m(DF_x|_{TN(x)})\leq \norm{DF_x|_{TN(x)}}< m(Df_{\pi(x)}|_{E^u_f} )
$$
then $F$ is partially hyperbolic and  a dynamically coherent. Moreover since the map in the base $f$ is a fibered lifted p.h. we have that $F$ has global product structure \ref{fph1} and center leaves in the universal cover are $\wt{\cWc_F}(\wt{x})=\wt{\cWc_f}(\pi(\wt{x}))\times \wt{N}$ proving \ref{fph2}. It is direct to check that the projection map $\wt{p}_F$ is just the composition $\wt{p}_f \circ \pi$, showing point \ref{fph3}. 

\subsection{Shadowing and stability} 

Let $f:M \to M$ be a \textit{fibered lifted partially hyperbolic diffeomorphism.} We are going to show some properties of $\wt{M}_c=\wt{M}/\wt{\cWc_f}$ and of $\wt{f}_c.$

From now on for simplicity, we are going to note by $$[\wt{x}]:=\wt{p}(\wt{x})\in \wt{M}_c \ \ \text{for every} \ \ \wt{x}\in \wt{M}.$$ Recall that $M=\wt{M}/\Gamma$ where $\gamma=\pi_1(M)$ acts on $\wt{M}$ by isometries. Then we can define an action of $\Gamma$ in $\wt{M}_c$ by the equation
$$ \gamma \cdot [\wt{x}]:=[\gamma \cdot \wt{x}]$$ Since for every $\gamma \in \Gamma$ we have that $\gamma \cdot \wt{\cW}^c_f(\wt{x})=\wt{\cW}^c_f(\gamma\cdot \wt{x})$ the action is well defined and moreover for every $\gamma\in \Gamma$ and every $\wt{x},\wt{y}\in \wt{M}$ we have:
\begin{eqnarray*}
	dist(\gamma \cdot [\wt{x}], \gamma \cdot [\wt{y}]) &=& dist( \gamma \cdot \wt{p}(\wt{x}), \gamma \cdot \wt{p}(\wt{y}) )=d_H(\gamma \cdot \wt{\cW}^c_f(\wt{x}), \gamma \cdot \wt{\cW}^c_f(\wt{y}))   \\
	&=&d_H(\wt{\cW}^c_f(\gamma \cdot \wt{x}), \wt{\cW}^c_f(\gamma \cdot \wt{y})) = d_H(\wt{\cW}^c_f(\wt{x}), \wt{\cW}^c_f(\wt{y})) \\
	&=&dist(\wt{p}(\wt{x}),\wt{p}(\wt{y}))=dist([\wt{x}],[\wt{y}])
\end{eqnarray*} and the action preserves the distance.

The following properties of $\wt{M}_c$ and of $\wt{f}_c$ are ``known'' in the examples above. Recall that for $[\wt{x}]$ we have the stable and unstable sets $\cWs_{\wt{f}_c}([\wt{x}])$ and  $\cWu_{\wt{f}_c}([\wt{x}])$ and the map $[[\wt{x}],[\wt{y}]]= \cWs_{\wt{f}_c}([\wt{x}])\cap  \cWu_{\wt{f}_c}([\wt{y}])$

\begin{prop}\label{p.propertiesc}
Let $f:M\to M$ be a fibered lifted partially hyperbolic diffeomorphism and let $\wt{M}_c, \wt{f}_c$ and 
$\wt{p}$ as above. Then:

\begin{enumerate}
\item Given $K>0$ there is $C>0$ such that, if $d(\wt{x},\wt{y})\leq K$ then $dist([\wt{x}],[\wt{y}])\leq C$. \label{p1properties}

\item Given $K>0$ there exists $C>0$ such that if $dist([\wt{x}],[\wt{y}])<K$ then $dist([\wt{x}],[[\wt{x}],[\wt{y}]])\le C.$ The same with $[\wt{y}].$ \label{p2properties}

\item The map $\wt{p}:\wt{\cW}^{\sigma}_f(\wt{x})\to \cW^{\sigma}_{\wt{f}_c} ([\wt{x}]), \sigma=s,u$ is a homeomorphism. In particular $\cWs_{\wt{f}_c}([\wt{x}])$ and  $\cWu_{\wt{f}_c}([\wt{x}])$ are homeomorphic to euclidean spaces, and so contractible. \label{p3properties}

\item $\wt{f}_c:\wt{M}_c\to\wt{M}_c$ is infinitely expansive: if $\sup_{n\in\Z}dist(\wt{f}_c^n([\wt{x}]),\wt{f}_c^n([\wt{y}]))<\infty$ then $[\wt{x}]=[\wt{y}].$ \label{p4properties}

\item $\wt{f}_c:\wt{M}_c\to\wt{M}_c$ has the \textit{global shadowing property}: given $K>0$ there exists $\alpha>0$ such that every $K$-pseudo orbit is $\alpha$-shadowed by a true orbit, i.e., if $[\wt{x}_n]$ is a sequence in $\wt{M}_c$ with $dist(\wt{f}_c([\wt{x}_n]),[\wt{x}_{n+1}])<K$ for all $n\in\Z$, there exists a unique $[\wt{y}]$ such that $dist(\wt{f}^n_c([\wt{y}]),[\wt{x}_n])<\alpha$ for all $n\in\Z.$ \label{p5properties}
\end{enumerate}
\end{prop}

\begin{proof}

Item \eqref{p1properties}. Let $D$ be a compact fundamental domain of $\wt{M}$ and let $\gamma_1,\dots,\gamma_k \in \Gamma$ be such that if $\wt{x}\in D$ then $B(\wt{x},K)\subset \bigcup_{i=1}^k \gamma_i \cdot D =: \widehat{D}$. By compactness and since $\wt{p}:\wt{M}\to \wt{M}/\wt{\cW}^c_f$ is continuous, there exists $C>0$ such that if $\wt{x},\wt{y} \in \widehat{D}$ then $dist(\wt{p}(\wt{x}),\wt{p}(\wt{y}))\leq C$. Now if $\wt{z},\wt{w}\in \wt{M}$ and $d(\wt{z},\wt{w})\leq K$, there is $\gamma$ such that $\gamma \cdot \wt{z}\in D$ and this implies that  $\gamma \cdot \wt{w}\in D$ too. We conclude that: $dist(\wt{p}(\wt{z}),\wt{p}(\wt{w}))=dist(\wt{p}(\gamma \cdot\wt{x}),\wt{p}(\gamma \cdot\wt{y})) \leq C$.
		
Item \eqref{p2properties} is very similar, we just prove the assertion with $[\wt{x}].$ Let $D$ be a compact fundamental domain and consider $\wt{p}(D)=D_1$ and let $\hat{D}$ be a compact set containing $B(D_1,K). $ By continuity and compactness there exists $C>0$ such that if $[\wt{x}],[\wt{y}]$ are in $\hat{D}$ then 	$dist([\wt{x}],[[\wt{x}],[\wt{y}]])\le C$. Now, for any $[\wt{x}]$ we have that there exists $\gamma\in\Gamma$ such that $[\gamma \cdot\wt{x}]\in D_1$. If $dist([\wt{x}],[\wt{y}])<K$ then $\gamma \cdot[\wt{x}]$ and $\gamma\cdot [\wt{y}]$ are in $\hat{D}$ and it holds too that $[\gamma \cdot \wt{x},\gamma \cdot \wt{y}]=\gamma \cdot [[\wt{x}],[\wt{y}]].$ Since $\gamma$ is an isometry on $\wt{M}_c$ we have that $dist([\wt{x}],[[\wt{x}],[\wt{y}]])\le C.$ 

Item \eqref{p3properties}. By the GPS we have that $\wt{p}:\wt{\cW}^{u}_f(\wt{x})\to \cWu_{\wt{f}_c}([\wt{x}])$ is injective. It is onto by definition. That $\wt{p}|_{\wt{\cW}^u_{\wt{f}_c([\wt{x}])}}$ is continuous is a consequence of the quotient topology. We have to prove that $\wt{p}$ is open ($\wt{p}^{-1}$ is continuous) when restricted to unstable leaves. Take $U$ an open set in $\wt{\cW}^u_f(\wt{x})$, then the  saturation of $U$ by center-stable leaves  is an open set $A$in $\wt{M}$ and moreover $U=A \cap \wt{\cW}^u_f(\wt{x})$. By the GPS we know that $A$ is also saturated by center leaves, which implies that $\wt{p}(A)$ is open in $\wt{M_c}$. Finally just observe that $\wt{p}(U)=\wt{p}(A)\cap \cW^u_{\wt{f_c}}([\wt{x}])$.


Item \eqref{p4properties} is also clear by the GPS and that points in $\cWu_{\wt{f}_c}([\wt{x}])$ approach to each other in the past and diverges to $\infty$ in the future and similar for $\cWs_{\wt{f}_c}([\wt{x}])$ interchanging past and future. 

Finally Item \eqref{p5properties} follows from the classical proof and Item \eqref{p2properties}. For the seek of completeness we are going to give the general lines. Let $K_1>0.$ We know by Item \eqref{p2properties} that there exists $C>0$ such that if $dist(\xc,\yc)<K_1$ then $[\xc,\yc]\in D^s_{\wt{f_c}}(\xc, C)\cap D^u_{\wt{f_c}}(\yc,C).$ Let $C_1>0$ such that if $dist(\xc,\yc)<K_1+C$ then $[\xc,\yc]\in D^s_{\wt{f_c}}(\xc, C_1)\cap D^u_{\wt{f_c}}(\yc,C_1).$ We may assume that $C_1>2C.$ 

Let $N>0$ be such that:
\begin{itemize}
\item $\fc^N(D^s_{\wt{f_c}}(\xc, 2C_1)\subset D^s_{\wt{f_c}}(\fc^N(\xc), C)$ 
\item $\fc^{-N}(D^u_{\wt{f_c}}(\xc,2C_1)\subset D^u_{\wt{f_c}}(\fc^{-N}(\xc), C)$
\end{itemize}
Let $\xc_{n\ge 0}$ be a $K_1$-pseudo orbit for $\fc^N.$ Define by induction $\zc_n$ as follows:
\begin{itemize}
\item $\zc_0=\xc_0$
\item $\zc_{n+1}=[\xc_{n+1},\fc^N(\zc_n)]=D^s_{\wt{f_c}}(\xc_{n+1}, C_1)\cap D^u_{\wt{f_c}}(\zc_n,C_1)$
\end{itemize}
It is not difficult to see (by induction) that for $0\le j\le n$ it holds that 
$$\fc^{-jN}(\zc_n)\in D^u_{\wt{f_c}}(\zc_{n-j}, C_1).$$
Setting $\yc_n:=\fc^{-nN}(\zc_n)$ we conclude that $dist(\fc^{jN}(\yc_n),\xc_j)\le 2C_1$ for $0\le j\le n. $ If $\yc$ is an accumulation point of $ \yc_n$ we have that $dist(\fc^{jN}(\yc),\xc_j)\le 2C_1$ for $j\ge 0.$ Now, let $K>0$ and $\xc_{k\ge 0}$ be a $K$-pseudo orbit for $\fc.$ Thus, there exists $K_1$ such that $\xc_{kN}$ is a $K_1$-pseudo orbit for $\fc^N.$ Let $\yc$ be as before, i.e., $dist(\fc^{jN}(\yc),\xc_{jN})\le 2C_1$ for $j\ge 0.$ Then, there exists $\alpha=\alpha(C_1)$ such that $dist(\fc^{j}(\yc),\xc_j)\le \alpha$ for $j\ge 0.$ We just shadowed future pseudo orbits. From this it is classical to shadow bi-infinite pseudo orbits. Since $\fc$ is infinitely expansive we get that the orbit is unique. 
\end{proof}

\begin{thm}[Stability of fibered lifted partially hyperbolic diffeomorphisms] \label{thmsta}
	Let $f:M\to M$ be a fibered lifted partially hyperbolic diffeomorphism. Then for every $g:M \to M$ such that $\sup\{dist(\wt{f}(\wt{x}),\wt{g}(\wt{x}))\}<K<\infty$ for some lift $\wt{g}:\wt{M}\to \wt{M}$, there exist a continuous and surjective map $H_g:\wt{M} \to \wt{M}_c$ and a number $\a=\a(f,K)>0$ such that:
	\begin{enumerate}
		\item $\wt{f}_c\circ H_g=H_g\circ \wt{g}$ \label{point1stability}
		\item $d_{C^0}(H_g,\wt{p})<\a$ \label{point2stability}
		\item the map $H_g$ varies continuously with $g$ in the $C^0$ topology. \label{point3stability}
		\item $H_g$ is $\Gamma$ invariant. \label{point4stability}
			\end{enumerate}
	
\end{thm}

\begin{proof} 
	Let $f$ be a fibered lifted partially hyperbolic diffeomorphism. Observe that if we set $H_f=\wt{p}$ we have $\wt{f}_c\circ H_f=H_f\circ \wt{f}.$ 
	
	Let $g\in \text{PH}(M)$ be such that $\sup\{dist(\wt{f}(\wt{x}),\wt{g}(\wt{x})): \wt{x}\in \wt{M}\}=K<\infty$ for some lift $\wt{g}:\wt{M} \to \wt{M}$ on the universal cover. Now for this $K>0$ we know by 	Item \ref{p1properties} of Proposition \ref{p.propertiesc} that there is $C>0$ such that, if $d(\wt{x},\wt{y})\leq K$ then $dist (\wt{p}(\wt{x}),\wt{p}(\wt{y}))\leq C$.
	
	Given a point $\wt{x}\in \wt{M}$  we define the following sequence:
	\begin{equation*}
		G_n(\wt{x}):= \wt{p}(\wt{g}^n(\wt{x}))=[\wt{g}^n(\wt{x})] 
	\end{equation*} 
	We claim that $\{G_n(\wt{x})\}_{n\in \Z}$ is a $C$-pseudo orbit with respect to $\wt{f}_c.$ First observe that:
	$$\wt{f}_c(G_n(\wt{x}))=\wt{f}_c  \circ \wt{p} (\wt{g}^n(\wt{x}))=\wt{p} \circ \wt{f}(\wt{g}^n(\wt{x}))
	$$
	Then we have that:
	\begin{eqnarray*}  dist(\wt{f}_c(G_n(\wt{x})),G_{n+1}(\wt{x})) 
			& = & dist(\wt{p} \circ \wt{f}(\wt{g}^n(\wt{x})),\wt{p}(\wt{g}^{n+1}(\wt{x})))\\
		& = & dist(\wt{p}(\wt{f} (\wt{g}^n(\wt{x}))), \wt{p} (\wt{g} (\wt{g}^n(\wt{x})))) \le C
		\end{eqnarray*} 
		 where the last inequality holds since  $d(\wt{f}(\wt{g}^n(\wt{x})),\wt{g}(\wt{g}^n(\wt{x})))\leq K$ and therefore $dist(\wt{p}(\wt{f} (\wt{g}^n(\wt{x}))), \wt{p} (\wt{g} (\wt{g}^n(\wt{x}))))\leq C$.

	We conclude that  $\{G_n(\wt{x})\}$ is a $C$-pseudo orbit with respect to $\wt{f}_c$. By the shadowing property (Item \ref{p5properties} of Proposition \ref{p.propertiesc}) we obtain a unique point  $[\wt{y}]$ such that $dist(\wt{f}^{n}_c([\wt{y}]), G_n(\wt{x}))<\a$ for every $n\in \Z$. Notice that $\a$ depends only on $f$, $C$ and $K$. Therefore the map $H_g:\wt{M} \to \wt{M}_c$ given by $H_g(\wt{x})=[\wt{y}]$ is well defined. 
	Now by definition we have: $$G_{n+1}(\wt{x})=\wt{p}(\wt{g}^{n+1}(\wt{x}))=\wt{p}(\wt{g}^n(\wt{g}(\wt{x})))=G_n(\wt{g}(\wt{x}))$$ 
	Then 
	$$dist(\wt{f}_c^n(\wt{f}_c \circ H_g)(\wt{x})),G_n(\wt{g}(\wt{x})))=dist(\wt{f}_c^{n+1}(H_g(\wt{x})),G_{n+1}(\wt{x}))<\a$$
	and the uniqueness in the shadowing property implies that
	$$H_g(\wt{g}(x))=\wt{f}_c(H_g(\wt{x})) \ \ \text{or equivalently} \ \  \wt{f}_c \circ H_g=H_g\circ \wt{g}$$
	proving Item \ref{point1stability}. By definition we have that  $dist(\wt{f}_c^n \circ H_g(\wt{x}),\wt{p}\circ \wt{g}^n(\wt{x}))<\a$, thus taking $n=0$ gives $dist(H_g(\wt{x}),\wt{p}(\wt{x}))<\a$ for every $\wt{x}\in \wt{M}$, proving Item \ref{point2stability}. 
	
	To see the continuity of $H_g$ suppose that the sequence $\{\wt{x_k}\}_{k\in \N} \subset \wt{M}$ is such that $\wt{x_k}\to \wt{x}$ as $k\to \infty$, and fix some integer $l\in\Z$. Then,
	\begin{eqnarray*}
		dist(\wt{f}_c^l(\lim_{k\to \infty} H_g(\wt{x_k})),\wt{p}(\wt{g}^{l}(\wt{x})))= dist(\wt{f}_c^l(\lim_{k\to \infty} H_g(\wt{x_k})),\wt{p}\circ \wt{g}^{l}(\lim_{k\to \infty} \wt{x_k}))\\
		=\lim_{k\to \infty} dist(\wt{f}_c^{l}(H_g(\wt{x_k})),\wt{p}(\wt{g}^{l}(\wt{x_k})))\le \a
	\end{eqnarray*}
	Since $l\in \Z$ is arbitrary, by the uniqueness of the shadowing we get $\lim_{k\to \infty} H_g(\wt{x_k})=H_g(\wt{x})$ and $H_g$ is continuous. Since $d_{C^0}(H_g,\wt{p})<\a$ by a degree argument we get that $H_g$ is surjective.
	
	To prove the continuous variation with respect to $g$, take some $\e>0$ and fix some large $N_0 \in \N$ such that if $dist([\wt{x}],[\wt{y}])<\e$ then $dist(\wt{f}_c^{N_0}([\wt{x}]), \wt{f}_c^{N_0}([\wt{y}])>2\a +C$ or $dist(\wt{f}_c^{-N_0}([\wt{x}]), \wt{f}_c^{-N_0}([\wt{y}])>2\a +C$. We can always find such $N_0$ since $\wt{f}_c$ is infinitely expansive by Item \ref{p4properties} of Prop \ref{p.propertiesc}. Let $\cU(g)$ be the $C^{0}$ neighbourhood of $g$ s.t. for every $g'\in \cU(g)$, $\wt{x}\in \wt{M}$ and $\abs{j}\leq N_{0}$ we have $d(\wt{g}^j(\wt{x}),\wt{g'}^j(\wt{x}))<K$. Now take $g'\in \cU(g)$, $\wt{x}\in \wt{M}$ and $\abs{j}\leq N_0$:
	\begin{eqnarray*}
		dist(\wt{f}_c^{j}(H_g(\wt{x})),\wt{f}_c^j(H_{g'}(\wt{x}))) &\leq& dist(\wt{f}_c^{j}(H_g(\wt{x})),\wt{p}(\wt{g}^j(\wt{x}))) \\ &+& dist(\wt{p}(\wt{g}^j(\wt{x})),\wt{p}(\wt{g'}^j(\wt{x})))\\
		&+& dist(\wt{p}(\wt{g'}^j(\wt{x})), \wt{f}_c^j(H_{g'}(\wt{x}))) \\
		&\leq& \a +C +\a = 2\a+C
	\end{eqnarray*}	
	where the first and third terms of the inequalities come from the shadowing property, and the second one since $d(\wt{g}^j(\wt{x}),\wt{g'}^j(\wt{x}))<K$. This implies $dist(H_g(\wt{x}),H_{g'}(\wt{x}))<\e$ by the above condition and therefore we get Item \ref{point3stability}. 
	
	To finish the proof we have to prove that $H_g$ is $\Gamma$-invariant. Recall that by definition we have $[\wt{x}]=\wt{p}(\wt{x})$ and $\gamma \cdot [\wt{x}]=[\gamma \cdot \wt{x}]$. First notice that if we call $\varphi:\Gamma\to \Gamma$ the induced map of $\wt{f}$ in the fundamental group, we get that for every $\gamma\in \Gamma$ and every $\wt{x}\in \wt{M}$:
	$$\wt{f}(\gamma \cdot \wt{x})=\varphi(\gamma)\cdot \wt{f}(\wt{x})$$ and the same happens for every $g$ as in the hypothesys: $\wt{g}(\gamma \cdot \wt{x})=\varphi(\gamma)\cdot \wt{g}(\wt{x})$. By induction we get that $\wt{f}^n(\gamma \cdot \wt{x})=\varphi^n(\gamma)\cdot \wt{f}^n(\wt{x})$. 
	In a similar way we have:
	\begin{eqnarray*}
		\wt{f}_c(\gamma \cdot [\wt{x}])&=& \wt{f}_c([\gamma \cdot \wt{x}])= \wt{f}_c (\wt{p}(\gamma \cdot \wt{x}))=\wt{p} \circ \wt{f}(\gamma \cdot \wt{x}) \\
		&=& \wt{p} (\varphi(\gamma)\cdot \wt{f}(\wt{x}))=\varphi(\gamma)\cdot  \wt{p}(\wt{f}(\wt{x}))   \\
		&=& \varphi(\gamma)\cdot \wt{f}_c \circ\wt{p}(\wt{x}) = \varphi(\gamma)\cdot \wt{f}_c([\wt{x}])
	\end{eqnarray*} By induction we get that $\wt{f}_c^n(\gamma \cdot [\wt{x}])=\varphi^n(\gamma)\cdot \wt{f}_c^n([\wt{x}])$. 
	Finally just observe that:
	\begin{eqnarray*} G_n(\gamma \cdot \wt{x})&=&\wt{p}(\wt{g}^n(\gamma \cdot \wt{x}))=\wt{p}(\varphi^n(\gamma)\cdot \wt{g}^n(\wt{x})) \\
		&=& \varphi^n(\gamma)\cdot \wt{p}(\wt{g}^n(\wt{x})) = \varphi^n(\gamma)\cdot G_n(\wt{x}) 
	\end{eqnarray*}  
	To sum up, for every $\gamma\in \Gamma$, $\wt{x}\in \wt{M}$ and $n\in \Z$ we have 
	\begin{eqnarray*}
		dist(\wt{f}_c^n(\gamma \cdot H_g(\wt{x})),G_n(\gamma \cdot \wt{x})) &=& dist(\varphi^n(\gamma)\cdot \wt{f}_c^n(H_g(\wt{x})),\varphi^n(\gamma)\cdot G_n(\wt{x}))\\
		&=&  dist(\wt{f}_c^n(H_g(\wt{x})),G_n(\wt{x}))< \a
	\end{eqnarray*} since the action by $\Gamma$ preserves the distance. By uniqueness of the shadowing property, we get that $H_g(\gamma \cdot \wt{x})=\gamma \cdot H_g(\wt{x})$, proving Item \ref{point4stability}.
\end{proof}
\begin{rem}
	As we mentioned above, in case $g=f$ if we set $H_f=\wt{p}$ we have $\wt{f}_c\circ H_f=H_f\circ \wt{f}$. In many parts of the article we will note $H_f$ instead of $\wt{p}$. 
\end{rem}

\begin{rem} \label{remisotopic}
	If $g\in \text{PH}(M)$ is isotopic to $f$ and we take a lift $\wt{g}$, then we always have that $$\sup \{d(\wt{f}(\wt{x}),\wt{g}(\wt{x})): \wt{x}\in \wt{M}\}<K<\infty
	$$
	therefore Theorem \ref{thmsta} applies and we get the map $H_g$.   
\end{rem}

\subsection{Sufficient conditions}

In section \ref{subsecexamples} we gave  some natural examples of fibered lifted partially hyperbolic diffeomorphisms and it was simple to verify the definition. In this section we will give some simple conditions to verify that a partially hyperbolic diffeomorphism is a fibered lifted one as we defined. Recall that $\wt{M}$ has a Riemannian metric induced by the metric on $M$. This metric also induces a Riemannian metric on each stable and unstable leaf of $\wt{f}$. We denote by
$$\wt{\cW}^{\sigma}_f(\wt{x}, R):=\{\wt{y}\in\wt{\cW}^{\sigma}_f(\wt{x}): d_{\wt{\cW}^{\sigma}_f}(\wt{x},\wt{y})<R\},\; \sigma=s,u$$
where $d_{\wt{\cW}^{\sigma}_f}$ is the leafwise distance. Let's start with a simple lemma.

\begin{lem}\label{l.prod}
Let $f:M\to M$ be a dynamically coherent partially hyperbolic diffeomorphism such that for the universal cover $\wt{f}:\wt{M}\to\wt{M}$ we have GPS as in \ref{fph1}. Then, given $K>0$ there exists $R>0$ such that 
for any $\wt{x}\in\wt{M}$ it holds
$$B(\wt{\cW}^{cs}_f(\wt{x}),K)\subset\bigcup_{\wt{y}\in \wt{\cW}^{cs}_f(\wt{x})}\wt{\cW}^{u}_f(\wt{y},R),\;\mbox{ and }\;B(\wt{\cW}^{cu}_f(\wt{x}),K)\subset\bigcup_{\wt{y}\in \wt{\cW}^{cu}_f(\wt{x})}\wt{\cW}^{s}_f(\wt{y},R).$$
In particular if $\wt{y}\in\wt{\cW}^s_f(\wt{x})$ then $d(\wt{f}^{-n}(\wt{y}),\wt{f}^{-n}(\wt{x}))\to_{n\to+\infty}\infty.$ The same for unstable leaves in the future.
\end{lem}

\begin{proof}
We will just prove for the center-stable manifolds, the other one is similar. Otherwise, there exist $K>0$, $R_n\to \infty$ and $\wt{x}_n,\wt{y}_n$ such that $dist(\wt{y}_n,\wt{\cW}^{cs}_f(\wt{x}_n))<K$ and $\wt{\cW}^u_f(\wt{y}_n,R_n)\cap \wt{\cW}^{cs}_f(\wt{x}_n)=\emptyset.$ Let $\wt{z}_n\in \wt{\cW}^{cs}_f(\wt{x}_n)$ be such that $dist(\wt{y}_n,\wt{z}_n)<K.$ Translating by an appropriate cover transformation $\gamma_n$ we may assume that $\wt{z}_n$ belongs to a fundamental domain, and (taking subsequence if necessary) we may assume that $\wt{y}_n,\wt{z}_n$ converge to $\wt{y},\wt{z}$ respectively.
It follows that $\wt{\cW}^u_f(\wt{y})\cap \wt{\cW}^{cs}_f(\wt{z})=\emptyset$ contradicting the Global Product Structure.

The last statement follows immediately. Indeed, $d_{\wt{\cW}^s_f}(\wt{f}^{-n}(\wt{y}),\wt{f}^{-n}(\wt{x}))\to_{n\to+\infty}\infty$ and if the ambient distance does not goes to infinity, there exist $K$ and $n_j\to\infty$ such that $d(\wt{f}^{-n_j}(\wt{y}),\wt{f}^{-n_j}(\wt{x}))<K$ for any $j$, implying that $\wt{f}^{-n_j}(\wt{y})\in \wt{\cW}^s_f(\wt{f}^{-n_j}(\wt{x}),R)$ for some $R>0,$ a contradiction.

\end{proof}

The next proposition gives conditions to a dynamical coherent partially hyperbolic diffeomorphism to be fibered lifted.

\begin{prop}\label{p.suficient}
Let $f:M\to M$ be a dynamically coherent partially hyperbolic diffeomorphism and let $\wt{f}:\wt{M}\to \wt{M}$ be its lift. Then, $f$ is fibered lifted provided:

\begin{enumerate}
\item  the foliations $\wt{\cW}^{cs}_f$ and $\wt{\cW}^{u}_f$ have global product structure; \\ the foliations $\wt{\cW}^{cu}_f$ and $\wt{\cW}^{s}_f$ have global product structure. 

\item Let $\wt{x}\in M$. Then there exists a center foliated neighborhood $U$ of $\wt{x}$ such that if $\wt{y}\in U$ then there is just one connected component of $\wt{\cW}^c_f(\wt{y}) \cap U$ and it is the plaque through $\wt{y}$.

\item Given any two center leaves $\wt{\cW}^c_f(\wt{x})$ and $\wt{\cW}^c_f(\wt{y})$ then, there exist $\delta>0$ and $K>0$ such that:
$$ d_H(\wt{\cW}^c_f(\wt{x}),\wt{\cW}^c_f(\wt{y}))<K \mbox{ and }B(\wt{\cW}^c_f(\wt{x}),\delta)\cap \wt{\cW}^c_f(\wt{y})=\emptyset.$$

\end{enumerate}

\end{prop}

\begin{proof}
We have to check properties \ref{fph1}, \ref{fph2} and \ref{fph3} of Definition \ref{dffph}. First notice that Item (1) is exactly \ref{fph1}. Conditions (2) and (3) naturally show that $\wt{M}_c:=\wt{M}/\wt{\cW}^c_f$ is a topological manifold and the Hausdorff distance between center leaves leads to a distance in $\wt{M}_c$ compatible with the topology proving \ref{fph2}. We have to check \ref{fph3}, that is, the quotient map $\wt{f_c}:\wt{M}_c\to\wt{M}_c$ is globally uniformly hyperbolic (Definition \ref{dfhyphom}) with
$$\cWs_{\wt{f}_c}([\wt{x}]):=\wt{p}(\wt{\cW}^{s}_f(\wt{x}))\;\;\mbox{ and }\;\;\;\cWu_{\wt{f}_c}([\wt{x}]):=\wt{p}(\wt{\cW}^{u}_f(\wt{x}))	$$
Clearly, item (1) implies the GPS in $\wt{M}_c$ between $\cWs_{\wt{f}_c}$ and $\cWu_{\wt{f}_c}$ proving condition \ref{c3hyphom} of Definition \ref{dfhyphom}. Item (3) also implies that given $r>0$ there exists $c>0$ such that if $\wt{x},\wt{y}\in\wt{M}$, $\wt{y}\in\wt{\cW}^{\sigma}_f(\wt{x}), \sigma=s,u$ and 
$d(\wt{x},\wt{y})<r$ then $d_H(\wt{\cW}^c_f(\wt{x}),\wt{\cW}^c_f(\wt{y}))<c.$ Condition \ref{c1hyphom} of Definition \ref{dfhyphom} follows directly. The previous lemma and item (3) yield that given $K_1$ there exist $R$ and  $R_1$ such that 
$$B(\cW^{s}_{\wt{f}_c}(\xc),K_1)\subset \wt{p}\left(\bigcup_{\wt{y}\in \wt{\cW}^{cs}_f(\wt{x})}\wt{\cW}^{u}_f(\wt{y},R)\right)\subset\bigcup_{\yc\in \cW^{s}_{\wt{f}_c}(\xc)}D^{u}_{\wt{f}_c}(\yc,R_1)$$
$$B(\cW^{u}_{\wt{f}_c}(\xc),K_1)\subset  \wt{p}\left(\bigcup_{\wt{y}\in \wt{\cW}^{cu}_f(\wt{x})}\wt{\cW}^{s}_f(\wt{y},R)\right)\subset\bigcup_{\yc\in \cW^{u}_{\wt{f}_c}(\xc)}D^{s}_{\wt{f}_c}(\yc,R_1)$$
These imply that condition \ref{c2hyphom} of Definition \ref{dfhyphom} is satisfied as well, since $\wt{p}:\wt{\cW}^\sigma_f(\wt{x})\to \cW^\sigma_{\wt{f}_c}(\xc)$ is a homeomorphism for any $\wt{x}.$ This finish the proof. 
\end{proof}

The next Corollary is a consequence of the above lemma and proposition and our definitions. It will be important to have in mind in Section \ref{sectionintegrability}. 

\begin{cor}\label{c.pre}
Let $f:M\to M$ be a fibered lifted partially hyperbolic diffeomorphism. Then, the  following hold:

\begin{enumerate}

\item Given $K_1$ there exist $R$ and $R_1$ such that 
$$B(\cW^{s}_{\wt{f}_c}(\xc),K_1)\subset \wt{p}\left(\bigcup_{\wt{y}\in \wt{\cW}^{cs}_f(\wt{x})}\wt{\cW}^{u}_f(\wt{y},R)\right)\subset\bigcup_{\yc\in \cW^{s}_{\wt{f}_c}(\xc)}D^{u}_{\wt{f}_c}(\yc,R_1)$$
$$B(\cW^{u}_{\wt{f}_c}(\xc),K_1)\subset  \wt{p}\left(\bigcup_{\wt{y}\in \wt{\cW}^{cu}_f(\wt{x})}\wt{\cW}^{s}_f(\wt{y},R)\right)\subset\bigcup_{\yc\in \cW^{u}_{\wt{f}_c}(\xc)}D^{s}_{\wt{f}_c}(\yc,R_1)$$

\item Given $r>0$ there exists $r_1$ such that 
$$\wt{p}^{-1}(D^\sigma_{\wt{f_c}}(\xc,r))\cap \wt{\cW}^\sigma_f(\wt{y})\subset \wt{\cW}^\sigma_f(\wt{y},r_1) \text{ for any }\wt{y} \text{ with }\yc=\xc, \sigma=s,u.$$ 

\item Given $K>0$ and $ \varepsilon>0$ there exists $R_2>0$ such that if $[\wt{x}], [\wt{y}]$ are in $B(\cW^s_{\wt{f}_c}([\wt{z}]),K)$ and $dist([\wt{x}], [[\wt{x}],[\wt{y}]])<\varepsilon$ then $[[\wt{x}],[\wt{y}]]\in D^u_{\wt{f_c}}([\wt{y}],R_2).$ 
\end{enumerate}
\end{cor}

\begin{proof}

Notice that (1) is a consequence of Lemma \ref{l.prod} as we have seen at the end of last proposition. Item (2) is a consequence of a succesive aplication of (1) since given $R_1$ there exist $K_2$ such that 
$$\bigcup_{\yc\in \cW^{s}_{\wt{f}_c}(\xc)}D^{u}_{\wt{f}_c}(\yc,R_1)\subset B(\cW^{s}_{\wt{f}_c}(\xc),K_2)$$
$$\bigcup_{\yc\in \cW^{u}_{\wt{f}_c}(\xc)}D^{s}_{\wt{f}_c}(\yc,R_1)\subset B(\cW^{u}_{\wt{f}_c}(\xc),K_2)$$
Item (3) is also a consequence of (1) since there exists $K_1>K$, $R$ and $R_1$ such that $$B(\xc,\varepsilon)\subset B(\cW^s_{\wt{f}_c}([\wt{z}]),K_1)\subset \wt{p}\left(\bigcup_{\wt{w}\in \wt{\cW}^{cs}_f(\wt{z})}\wt{\cW}^{u}_f(\wt{w},R)\right)\subset\bigcup_{[\wt{w}]\in \cW^{s}_{\wt{f}_c}([\wt{z}])}D^{u}_{\wt{f}_c}([\wt{w}],R_1)$$ 
hence
$[[\wt{x}],[\wt{y}]]$ and $[\wt{y}]$ belong to $D^u_{\wt{f_c}}([\wt{w}],R_1)$ for some $[\wt{w}]\in {\cW}^{s}_{\wt{f_c}}([\wt{z}]).$ Taking $R_2=2R_1$ gives the lemma.




\end{proof}

\subsection{Re statement of the main results} \label{smainresults}

From now on $f\in \text{PH}(M)$ will be a fibered lifted partially hyperbolic diffeomorphism and we are going to consider the subset $\text{PH}_f(M)\subseteq \text{PH}(M)$ of partially hyperbolic diffeomorphisms such that:
\begin{equation*}
	\text{PH}_f(M)=\left\{ 
	\begin{array}{cc}
		& g \in \text{PH}(M) \ \textnormal{which are isotopic to} \ f \ \text{and such that} \ \ \ \\
		& \textnormal{dim}E^{\sigma}_g=\textnormal{dim}E^{\sigma}_f, \ \textnormal{for} \ \sigma=s,c,u
	\end{array}
	\right\}
\end{equation*} 

By Theorem \ref{thmsta} (and Remark \ref{remisotopic}) we have that for every $g\in \text{PH}_f(M)$ there is a continuous and surjetive map $H_g:\wt{M}\to \wt{M}_c$ such that $\wt{f}_c \circ H_g=H_g\circ \wt{g}$. The first direct consequence of this semiconjugacy is the following:
$$\text{if} \ \  \wt{y}\in \wt{\cW}^{s}_g(\wt{x}) \ \ \text{then} \ \  H_g(\wt{y})\in\cWs_{\wt{f}_c}(H_g(\wt{x}))$$ and the same happens with the unstable manifold: 
$$\text{if} \ \ \wt{y}\in \wt{\cW}^{u}_g(\wt{x}) \ \ \text{then} \ \ H_g(\wt{y})\in\cWu_{\wt{f}_c}(H_g(\wt{x}))$$ 
This is easy to see since $\wt{y}\in \wt{\cW}^{s}_g(\wt{x})$ then $d(\wt{g}^n(\wt{y}),\wt{g}^n(\wt{x}))\to 0$ for $n\to +\infty$. This implies that 
$dist(H_g\circ \wt{g}^n(\wt{y}),H_g \circ \wt{g}^n(\wt{x}))\to 0$ and by the semiconjugacy relation this is the same as $dist(\wt{f}_c^n (H_g(\wt{y})),\wt{f}_c^n(H_g(\wt{x})))\to 0$. By (topological) hyperbolicity this can only happen if $H_g(\wt{y})\in\cWs_{\wt{f}_c}(H_g(\wt{x}))$. The same calculation works for the past.

On the other hand suppose there are points $\wt{x},\wt{y} \in \wt{M}$ such that their orbits are at finite distance at any time (this is the ``ideal'' picture of the behavior on center leaves), then since $\wt{f}_c$ is uniformly hyperbolic we have $H_g(\wt{x})=H_g(\wt{y})$. This motivates the following definition, which is the analogous to the one introduced in \cite{FPS}.  

\begin{df} We say that a dynamically coherent $g\in \text{PH}_f(M)$ is \textit{center-fibered} (CF) if $H_g^{-1}(H_g(\wt{x}))=\wt{\cW}^c_g(\wt{x})$ for every $\wt{x}\in \wt{M}$. 
\end{df}
In particular this means that two different center leaves of $\wt{g}$ are sent by $H_g$ to two different points in $\wt{M}_c$.

Now given a fibered lifted partially hyperbolic diffeomorphism $f$, we are going to note:
\begin{equation*}
	\text{PH}_f^0(M)=\left\{ 
	\begin{array}{cc}
		& \textnormal{connected componentes of $\text{PH}_f(M)$} \ \textnormal{which contains a} \ \ \\
		& \textnormal{DC and CF p.h.d. with global product structure} \ \ 
	\end{array}
	\right\}
\end{equation*}
We remark that the fibered lifted partially hyperbolic diffeomorphism $f$ is itself center fibered by definition because $H_f=\wt{p}$, then for every $\wt{x}\in \wt{M}$ we have that: 
$$ (H_f)^{-1}(H_f)(\wt{x})=\wt{p}^{-1}(\wt{p}(\wt{x}))=\wt{\cW}^c_f(\wt{x})
$$
Then the set $\text{PH}_f^0(M)$ is a non-empty open set with at least one connected component. Let us mention that in \cite{FG} it is proved that given a linear Anosov $A:\T^d\to \T^d$ (with $d\geq 10$), the space of Anosov diffeomorphisms homotopic to $A$ has infinitely many connected components. In particular, this implies that $\text{PH}_f^0(M)$ may have more than one connected component (besides the one containing $f$). 

With these new notations we can restate the main results of the article. 

\begin{thm} \label{teoA}
	Every $g\in \textnormal{PH}_f^0(M)$ is dynamically coherent and center-fibered.
\end{thm}

A direct consequence from the proof of this theorem, is that it implies to have plaque expansiveness in the whole connected component. Applying Theorem \ref{thmplaqueexpansivetostablydc0} and a connectedness argument, we can obtain the following classification result.

\begin{thm} \label{teoB}
	Any two diffeomorphisms in the same connected component of $\textnormal{PH}_f^0(M)$ are leaf conjugate. In particular every $g\in \textnormal{PH}_f^0(M)$ in the same connected component of $f$ is leaf conjugate to $f$.
\end{thm}

\section{Integrability for fibered lifted partially hyperbolic diffeomorphisms} \label{sectionintegrability}

In this section we are going to see an integrability criterion for partially hyperbolic diffeomorphisms isotopic to fibered lifted partially hyperbolic diffeomorphisms. For this criterion, we apply the ideas of \cite{FPS} to our setting of quotient dynamics.

\subsection{$\sigma$-Properness}

Recall that given $g\in \text{PH}_f(M)$, for any $\sigma\in \{s,u\}$, for any $\wt{x} \in \wt{M}$, and for any  $\e>0$, we denote by $$\wt{\cW}_g^{\sigma}(\wt{x},\e):=\{\wt{y} \in \wt{\cW}_g^{\sigma}(\wt{x}): d_{\wt{\cW}_g^{\sigma}}(\wt{x},\wt{y})< \e \}$$ to the $\e$-ball in $\wt{\cW}_g^{\sigma}$ of center $\wt{x}$ and radius $\e$, where $d_{\wt{\cW}_g^{\sigma}}$ denotes the leafwise distance, that is the distance induced by the Riemannian metric in $\wt{M}$ restricted to the leaves. In the same way, recall that for $[\wt{x}]\in \wt{M}_c$ and $r>0$ we denote by\footnote{Observe that we don't know whether $\wt{M}_c$ is a riemannian manifolds and neither we know if $\cW^\sigma_{\wt{f}_c}$ are smooth manifolds; for this we use the amibent metric.}  $$D^\sigma_{\wt{f}_c}([\wt{x}], r)=\{[\wt{y}]\in \cW^\sigma_{\wt{f}_c}([\wt{x}]): dist([\wt{x}],[\wt{y}])<r\}$$ for $\sigma=s,u$. From now on for simplicity, we are going to omit the subindex $\wt{f}_c$, i.e. we are going to note by $D^\sigma([\wt{x}], r)$ instead of $D^\sigma_{\wt{f}_c}([\wt{x}], r)$.


Recall that $\wt{p}: \wt{\cW}_f^{\sigma}(\wt{x})\to\cW^{\sigma}_{\wt{f}_c}([\wt{x}])$ is a homeomorphism for any $\wt{x}$. From  Corollary \ref{c.pre}, given $R>0$ there exists $R_1$ such that
$$\wt{p}^{-1}(D^\sigma([\wt{x}], R))\cap \wt{\cW}^{\sigma}_f(\wt{y})\subset \wt{\cW}^{\sigma}_f(\wt{y},R_1) \;\mbox{ for any } \wt{y} \;\mbox{ with } [\wt{y}]=[\wt{x}], \sigma=s,u.$$ 

The following definition is inspired by the one introduced in \cite{FPS}.

\begin{df}For $\sigma=s,u$ we say that $g\in \text{PH}_f(M)$ is \textit{$\sigma$-proper} if for every $\wt{x}\in \wt{M}$ the map $H_g$ restricted to $\wt{\cW}^{\sigma}_g(\wt{x})$ is uniformly proper. More precisely, for every $R>0$ there exists $R'>0$ such that
$$(H_g)^{-1}(D^{\sigma}(H_g(\wt{x}),R))\cap \wt{\cW}_g^{\sigma}(\wt{x}) \subset \wt{\cW}^{\sigma}_g(\wt{x},R') \ \ \text{for every} \ \  \wt{x}\in \wt{M}$$ 
\end{df}
\begin{rem} \label{osp1}
	In the previous definition, we can take $R=1$ by uniform hyperbolicity of the strong bundles, and the cocompactness of $\wt{M}$. 
\end{rem}
The definition of $\sigma$-properness can be expressed in a different and more geometric way. The next lemma gives the desire equivalence.  Given $g\in \text{PH}_f(M)$ we say that $g$ has condition: 
\begin{itemize}
	\item[$(I^{\sigma})$] If the function $H_g$ is injective restricted to $\wt{\mathcal{W}}^{\sigma}_g$-leaves.
	\item[$(S^{\sigma})$] If the function $H_g$ is surjective restricted to $\wt{\mathcal{W}}^{\sigma}_g$-leaves.
\end{itemize} 
Then if $g$ verifies both conditions, the map $H_g|_{\wt{\mathcal{W}}^{\sigma}_g(\wt{x})}:\wt{\mathcal{W}}^{\sigma}_g(\wt{x})\to \cW^{\sigma}_{\wt{f}_c}(H_g(\wt{x}))$ is a homeomorphism (recall that $\wt{\cW}^{\sigma}_g$ and $\cW^\sigma_{\wt{f}_c}$  are homeomorphic to euclidean spaces of the same dimension). 

\begin{lem} \label{equipro}
	If $g\in \textnormal{PH}_f(M)$ then, $g$ is $\sigma$-proper if and only if $g$ satisfies properties $(I^{\sigma})$ and $(S^{\sigma})$. Moreover $(I^{\sigma})$ implies $(S^{\sigma})$.
\end{lem}

\begin{proof}
We omit details since it is almost the same as Lemmas 3.2 and 3.4 in \cite{FPS}. The only if part is rather simple and follows form the fact that $\sigma$-properness implies injectivity, and since both spaces are homeomorphic to euclidean spaces of the same dimension and ``hyperbolicity'', it follows surjectivity.  For the if part, consider the function $\varphi:\wt{M}\to \R$ where
$$\varphi(\wt{x})=\inf\{R>0: (H_g)^{-1}(D^\sigma(H_g(\wt{x}),1))\cap \wt{\cW}_g^{\sigma}(\wt{x}) \subset \wt{\cW}^{\sigma}_g(\wt{x},R)\; \forall \;\wt{y}\in H_g(\wt{x})\}. $$
The function $\varphi$ is upper continuos and $\Gamma$-periodic, thus uniformly bounded. 
\end{proof}

\subsection{Strong almost dynamically coherence}

Recall that given a subset $K \subset \wt{M}$ and $R>0$ we call $B(K,R)$ the $R$-neighbourhood of $K$, that is:
$$B(K,R)=\{\wt{x}\in \wt{M}: \text{there is} \ \wt{y}\in K \ \text{s.t.} \  d(\wt{x},\wt{y})<R\}
$$ This includes the case $K=\wt{x}$ and $B(\wt{x},R)=\{\wt{y}\in \wt{M}:   d(\wt{x},\wt{y})<R\}$. 

\begin{df}[Almost parallel foliations] Given $\wt{\mathcal{F}}_1$ and $\wt{\mathcal{F}}_2$ two foliations in $\wt{M}$ we say they are \textit{almost parallel} if there exists $R>0$ such that for every $\wt{x}\in \wt{M}$, there are points $\wt{x_1}, \wt{x_2} \in \wt{M}$ such that:
	\begin{itemize}
		\item $\wt{\cF}_1(\wt{x})\subset B(\wt{\cF}_2(\wt{x_1}),R)$ and $\wt{\cF}_2(\wt{x_1})\subset B(\wt{\cF}_1(\wt{x}),R)$
		\item  $\wt{\cF}_2(\wt{x})\subset B(\wt{\cF}_1(\wt{x_2}),R)$ and $\wt{\cF}_1(\wt{x_2})\subset B(\wt{\cF}_2(\wt{x}),R)$  
	\end{itemize}
\end{df}
It's easy to see that this is an equivalence relation. Moreover the condition can be expressed in terms of the Hausdorff distance:
for every $\wt{x}\in \wt{M}$, there exist $\wt{x_1}, \wt{x_2}\in \wt{M}$ such that $d_H(\wt{\cF}_1(\wt{x}),\wt{\cF}_2(\wt{x_1}))<R$ and $d_H(\wt{\cF}_2(\wt{x}),\wt{\cF}_1(\wt{x_2}))<R$.

\begin{df} We say that $g\in \text{PH}_f(M)$ is \textit{strongly almost dynamically coherent} \textnormal{(SADC)} if there exist foliations $\mathcal{F}^{cs}_g$, $\mathcal{F}^{cu}_g$ (not necessarily invariant) such that:
	\begin{itemize}
		\item $\mathcal{F}^{cs}_g$, $\mathcal{F}^{cu}_g$ are transverse to $E^{u}_g$, $E^{s}_g$ respectively,
		
		\item $\wt{\mathcal{F}}^{cs}_g$, $\wt{\mathcal{F}}^{cu}_g$ are almost parallel to the foliations $\wt{\cW}^{cs}_f$, $\wt{\cW}^{cu}_f$ respectively.
	\end{itemize}	 
\end{df}

\begin{figure}[H]
	\begin{center}
		\includegraphics [width=14.5cm]{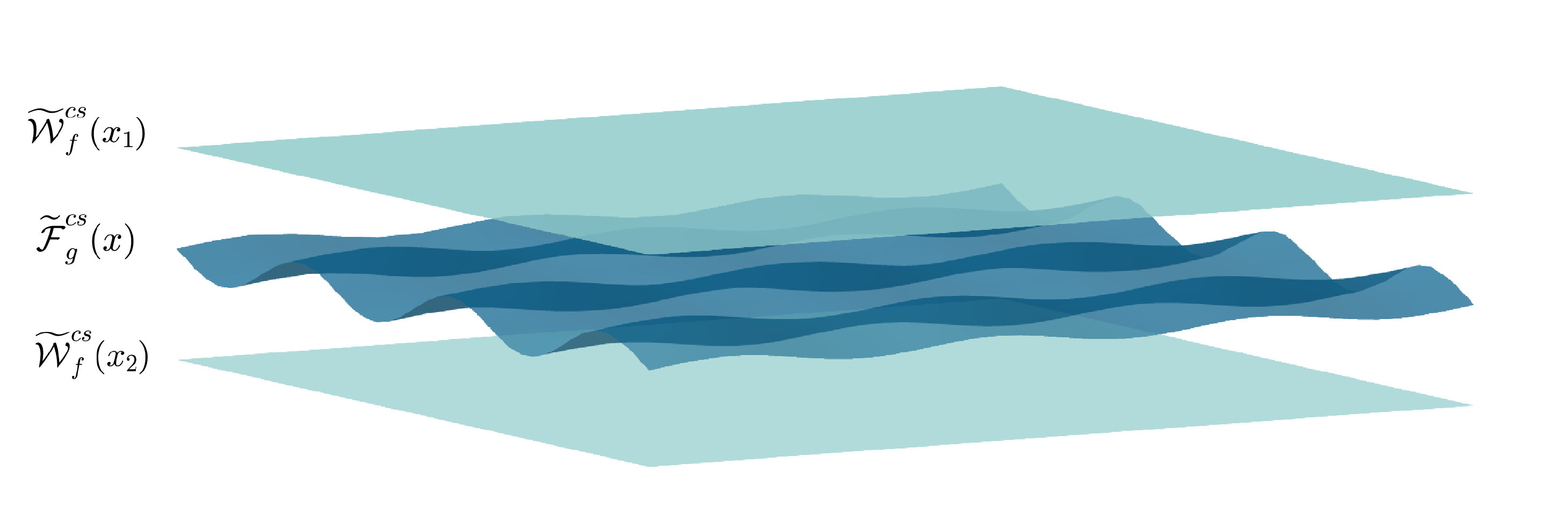}
		\caption{Almost parallel foliations $\mathcal{F}^{cs}_g$ and $\mathcal{W}^{cs}_f$}
	\end{center}
\end{figure}
\vspace{-0.5cm}

\begin{df} Given $g\in \text{PH}_f(M)$ which is SADC with their corresponding foliations $\mathcal{F}^{cs}_g$ and $\mathcal{F}^{cu}_g$, we say that $g$ has global product structure if $\mathcal{F}^{cs}_g$ and $\wt{\cW}^{u}_g$ have global product structure (GPS) and, $\mathcal{F}^{cu}_g$ and $\wt{\cW}^{s}_g$ have global product structure.
\end{df}

\subsection{Integrability criterion}

The following is an integrability criterion for partially hyperbolic diffeomorphisms isotopic to fibered lifted p.h. diffeomorphisms. For this criterion we adapt the ideas of \cite{FPS} to our setting via quotient dynamics. 

\begin{thm} \label{tcdi}
	Assume that $g\in \textnormal{PH}_f(M)$ verifies the following conditions:
	\begin{itemize}
		\item $g$ is $u$-proper. 
		\item $g$ is \textnormal{SADC} with global product structure. 
	\end{itemize}
	Then the bundle $E^{cs}_g$ is integrable into a $g$-invariant foliation $\mathcal{W}^{cs}_g$ that verifies $H_g^{-1}(\cW^{s}_{\wt{f}_c}(H_g(\wt{x})))=\wt{\cW}^{cs}_g(\wt{x})$. 
	Moreover $\wt{\mathcal{W}}^{cs}_g$ and $\wt{\cW}^{u}_g$ have global product structure.
\end{thm}

\begin{proof}
	The idea of the proof is pretty direct: take the foliation $\wt{\cF}^{cs}_g$ given by the SADC property and iterate it backwards by $\wt{g}$ hoping that in the limit it will converge to the desired foliation. Specifically the goal is to show that $\{(H_g)^{-1}(\cW^{s}_{\wt{f}_c}(H_g(\wt{y}))): \wt{y}\in \wt{M} \}$ is the center-stable foliation of $\wt{g}$. 
	
	First observe that this partition of $\wt{M}$ is $\wt{g}$-invariant: 
	\begin{eqnarray*}
		\wt{g}^{-1}((H_g)^{-1}(\cW^{s}_{\wt{f}_c}(H_g(\wt{y}))))&=&(H_g \circ \wt{g})^{-1}(\cW^{s}_{\wt{f}_c}(H_g(\wt{y})))= (\wt{f}_c\circ H_g)^{-1}(\cW^{s}_{\wt{f}_c}(H_g(\wt{y}))) \\
		&=& H_g^{-1}( \wt{f}_c^{-1}(\cW^{s}_{\wt{f}_c}(H_g(\wt{y}))))=H_g^{-1}(\cW^{s}_{\wt{f}_c}(\wt{f}_c^{-1}(H_g(\wt{y}))))	\\
		&=& H_g^{-1}(\cW^{s}_{\wt{f}_c}(H_g(\wt{g}^{-1}(\wt{y}))))
	\end{eqnarray*}
	Moreover the partition is invariant by deck translations since $H_g$ is $\Gamma$-invariant. Now take the foliation $\wt{\cF}^{cs}_g$ given by the SADC property. Since it is almost parallel to $\wt{\cW}^{cs}_f$ and $H_g$ is at bounded Hausdorff distance from $H_f=\wt{p}$ we have that $H_g(\wt{\cF}^{cs}_g(\wt{x}))$ is also at bounded Hausdorff distance from $\cW^s_{\wt{f}_c}([\wt{x}_1])$ for some $\wt{x}_1\in \wt{M}$.
	
		Since $\wt{\cW}^{u}_g$ and $\wt{\cF}^{cs}_g$ have global product structure, we can see the leaves of $\wt{\cF}^{cs}_g$ (and then of $\wt{g}^{-n}(\wt{\cF}^{cs}_g)$) as graphs of functions from $\R^{cs}$ to $\R^{u}$. Since the foliation $\wt{\cF}^{cs}_g$ is uniformly transversal to $E_g^{u}$ we know there are local product structure boxes of uniform size in $\wt{M}$, i.e. there is $\e>0$ s.t. $\forall \wt{x} \in \wt{M}$ there is a neighbourhood $V_{\wt{x}}\supseteq B(\wt{x},\e)$ and $C^{1}$-local coordinates $\psi_{\wt{x}}:\D^{cs}\times \D^{u} \to V_{\wt{x}}$ such that:
	\begin{itemize}
		\item $\psi_{\wt{x}}(\D^{cs}\times \D^{u})=V_{\wt{x}}$
		\item For every $\wt{y} \in B(\wt{x},\e)\subseteq V_{\wt{x}}$ we have that if we call $W^{\wt{x}}_{n}(\wt{y})$ to the connected component of $V_{\wt{x}}\cap \wt{g}^{-n}(\wt{\cF}^{cs}_g(\wt{g}^n(\wt{y})))$ that contains $\wt{y}$ then
		$$\psi_{\wt{x}}^{-1}(W^{\wt{x}}_n(\wt{y}))=\textnormal{graph}(h^{\wt{x},\wt{y}}_n)$$ 
		where $h^{\wt{x},\wt{y}}_n:\D^{cs}\to \D^{u}$ is a $C^{1}$ function with bounded first derivatives.  
	\end{itemize}
	In this way we get that the set $\{h^{\wt{x},\wt{y}}_n\}_{n\in \N}$ is precompact in the space of Lipschitz functions $\D^{cs}\to \D^{u}$ (\cite{HPS}). Therefore the leaves of $\wt{g}^{-n}(\wt{\cF}^{cs}_g)$ have convergent sub-sequences. From this point we have to deal with two problems: the first one is that \textit{a priori} there could be a leaf with more than one limit, and second, that in the limit, different leaves might merge. We will handle these two problems in the same way. For every $\wt{y} \in B(\wt{x},\e)$, we call $\cJ^{\wt{x}}_{\wt{y}}$ to the set of indices such that for every $\alpha \in \cJ^{\wt{x}}_{\wt{y}}$ there is a Lipschitz function $h^{\wt{x},\wt{y}}_{\infty,\alpha}:\D^{cs}\to \D^{u}$ and a subsequence $n_j\to +\infty$ such that: 
	$$h^{\wt{x},\wt{y}}_{\infty,\alpha}=\lim_{j\to +\infty}h^{\wt{x},\wt{y}}_{n_j}$$
	Every $h^{\wt{x},\wt{y}}_{\infty,\alpha}$ has its corresponding graph, and we note $W^{\wt{x}}_{\infty,\alpha}(\wt{y})$ to the image by $\psi_{\wt{x}}$ of this graph. 
	The following claim is crucial for the theorem. 
	\begin{claim}
		For every $\wt{z}\in B(\wt{x},\e)$ and every $\alpha \in \cJ^{\wt{x}}_{\wt{z}}$, we have that $H_g(W^{\wt{x}}_{\infty,\alpha}(\wt{z}))\subseteq \cW^s_{\wt{f}_c}(H_g(\wt{z}))$. 
		\begin{proof}
			Take $\wt{z}\in B(\wt{x},\e)$ and $\alpha \in \cJ^{\wt{x}}_{\wt{z}}$. Then by hypothesis there is subsequence $n_j\to +\infty$ such that $W^{\wt{x}}_{n_j}(\wt{z})\to W^{\wt{x}}_{\infty,\alpha}(\wt{z})$. Given $\wt{y}\in W^{\wt{x}}_{\infty,\alpha}(\wt{z})$ we want to prove that $H_g(\wt{y})\in \cW^s_{\wt{f}_c}(H_g(\wt{z}))$. Call $\wt{z}_{n_j}=\wt{\cW}_g^{u}(\wt{y})\cap W^{\wt{x}}_{n_j}(\wt{z})$ (see Figure \ref{figmerge} below). Then $\wt{z}_{n_j}\to \wt{y}$ when $j \to +\infty$ and $\wt{g}^{n_j}(\wt{z}_{n_j})\in \wt{\cF}^{cs}_g(\wt{g}^{n_j}(\wt{z}))$. If $H_g(\wt{y})=H_g(\wt{z})$ we're done. Suppose by the contrary that $H_g(\wt{z})\neq H_g(\wt{y})$, and they're not in the same stable leaf. 
			Then $H_g(\wt{z}_{n_j})\to H_g(\wt{y})\neq H_g(\wt{z})$. Note that $\wt{g}^{n_j}(\wt{z})$ and $\wt{g}^{n_j}(\wt{z}_{n_j})$ belong to the same leaf of $\wt{\cF}^{cs}_g$. Since $\cF^{cs}_g$ is almost parallel to $\wt{\cW}^{cs}_f$ and $H_g$ is $C^{0}$-close to $H_f=\wt{p}$, we know that there exists $K>0$ such that $H_g(\wt{g}^{n_j}(\wt{z}))$ and $H_g(\wt{g}^{n_j}(\wt{z}_{n_j}))$ belong to $B(\cW^s_{\wt{f}_c}([\wt{t_j}]), K)$ for some $\wt{t_j}$ and  any $j.$ Take $\varepsilon>0$ and the corresponding $R_1>0$ from Corollary \ref{c.pre} (3) for these $K$ and $\varepsilon$. Let $[\wt{w}]=[H_g(\wt{z}), H_g(\wt{y})]$, then we have $dist(\wt{f}^n_c([\wt{w}]), \wt{f}^n_c(H_g(\wt{z})))\to_n 0$. Thus for $j$ large, we have:
			\begin{itemize}
				\item  $H_g(\wt{g}^{n_j}(\wt{z}))$ and $H_g(\wt{g}^{n_j}(\wt{z}_{n_j}))$ belong to $B(\cW^s_{\wt{f}_c}([\wt{t_j}]), K)$
				\item $dist(\wt{f}^{n_j}_c(H_g(\wt{z})), \wt{f}_c^{n_j}([\wt{w}]))<\varepsilon$
			\end{itemize}
			Notice that by the semiconjugacy relation we have that $H_g(\wt{g}^{n_j}(\wt{z}))=\wt{f}^{n_j}_c(H_g(\wt{z}))$, thus by the corollary we just mentioned, we have that $ \wt{f}^{n_j}_c([\wt{w}]) \in D^u_{\wt{f}_c}(\wt{f}^{n_j}_c(H_g(\wt{z}_{n_j}),R_1))$ which in particular implies $dist(\wt{f}^{n_j}_c([\wt{w}]), \wt{f}^{n_j}_c(H_g(\wt{z}_{n_j}))<R_2$ for every $j$ large enough. On the other hand, since $dist([\wt{w}],H_g(\wt{z}_{n_j}))>\delta$ for some $\delta>0$ (recall that $H_g(\wt{z}_{n_j})\to H_g(y)\neq [\wt{w}]$) we have that  $dist(\wt{f}^{n}_c([\wt{w}]), \wt{f}^{n}_c(H_g(\wt{z}_{n_j}))\to +\infty$ which is a contradiction.
				
		\end{proof}
	\end{claim}
	We are going to solve the two problems mentioned above in the same way. Suppose first that $\wt{z} \in B(\wt{x},\e)$ has two different limits $W^{\wt{x}}_{\infty,\alpha}(\wt{z})$ and $W^{\wt{x}}_{\infty,\beta}(\wt{z})$. Then there are points $\wt{z_1}\in W^{\wt{x}}_{\infty,\alpha}(\wt{z})$ and $\wt{z_2}\in W^{\wt{x}}_{\infty,\beta}(\wt{z})$ that belong to the same $\wt{\cW}^{u}_g$-leaf. The previous claim implies that $H_g(\wt{z_1})$ and $H_g(\wt{z_2})$ belong to $\cW^{s}_{\wt{f}_c}(H_g(\wt{z}))$ and this can happen if and only if $H_g(\wt{z_1})=H_g(\wt{z_2})$ which contradicts the injectivity of $H_g|_{\wt{\mathcal{W}}^{u}_g}$.
	
	For the second problem we manage the same way. Let's suppose there are points $\wt{z_1} \neq \wt{z_2}$ in $B(\wt{x},\e)$ such that their limits $W^{\wt{x}}_{\infty,\alpha}(\wt{z_1})$ and $W^{\wt{x}}_{\infty,\beta}(\wt{z_2})$ have non empty intersection. Then we get two points $\wt{y_1} \in W^{\wt{x}}_{\infty,\alpha}(\wt{z_1})$ and $\wt{y_2} \in W^{\wt{x}}_{\infty,\beta}(\wt{z_2})$ inside the same $\wt{\cW}^{u}_g$-leaf. Again the previous claim implies $H_g(\wt{z_1})=H_g(\wt{z_2})$ and this contradicts the injectivity of $H_g|_{\wt{\mathcal{W}}^{u}_g}$.  
	\begin{figure}[H] 
		\begin{center}
			\includegraphics [width=13.5cm]{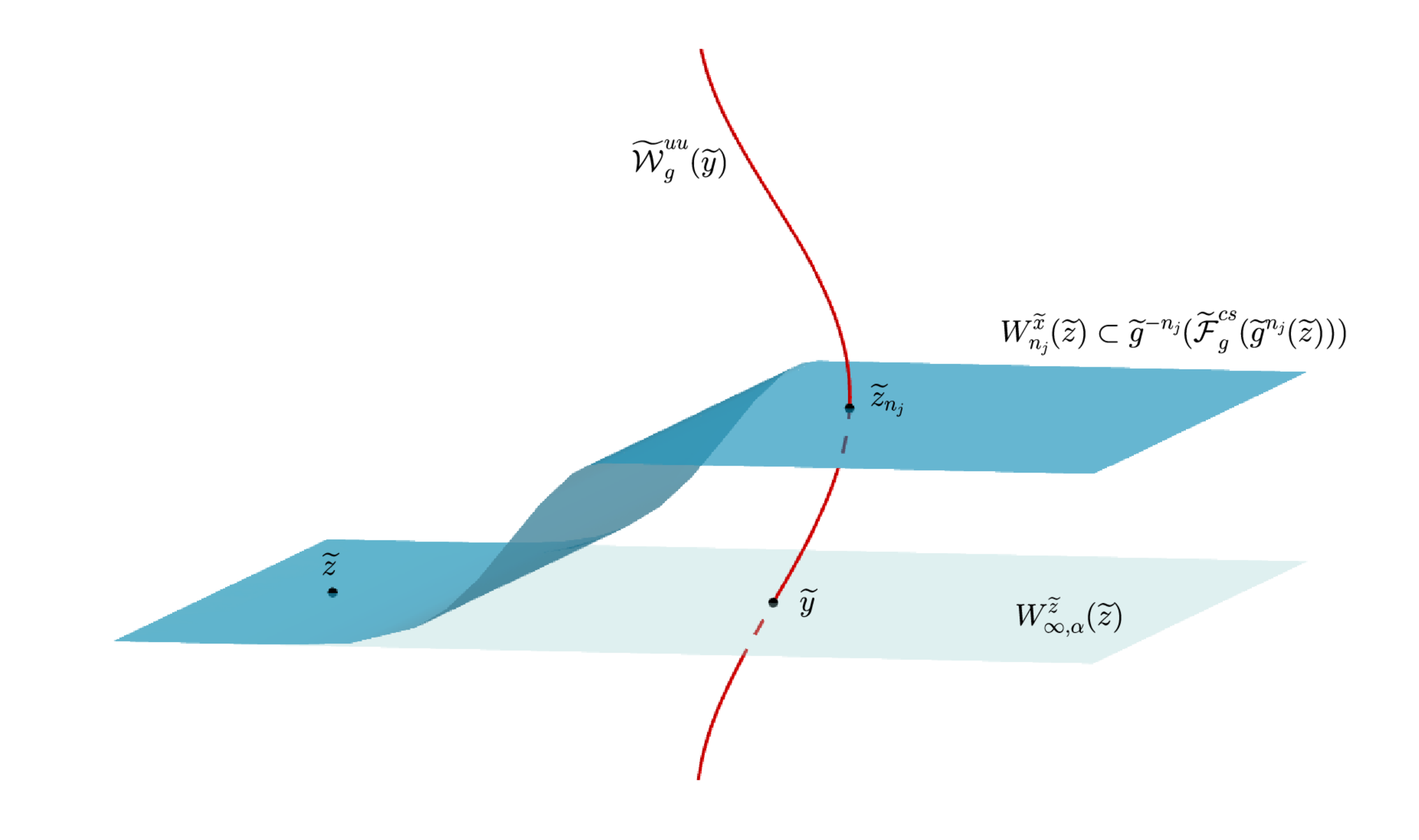}
			\caption{Plaques does not merge.} \label{figmerge}
		\end{center}
	\end{figure}
	\vspace{-0.5cm} 	
	To sum up, we obtained that for every $\wt{x}\in \wt{M}$ and every $\wt{y}\in B(\wt{x},\e)$, the limit $W^{\wt{x}}_{\infty}(\wt{y})$ of the $W^{\wt{x}}_n(\wt{y})$ leaves is unique, and for every pair of points $\wt{y},\wt{z} \in B(\wt{x},\e)$, their limits are disjoint or coincide. These limits are also $\wt{g}$-invariant. To get that it is truly a foliation, it's enough to observe the following: given two points $\wt{z},\wt{w}\in B(\wt{x},\e)$, we have that $W^{\wt{x}}_{\infty}(\wt{z})$ and $\wt{\cW}^{u}_g(\wt{w})$ intersect in a unique point. Since the leaves of $\wt{\cW}^{u}_g$ varies continuously and the plaques of $W^{\wt{x}}_{\infty}$ either coincide or are disjoint, we get a continuous function from $\D^{cs}\times \D^{u}$ to a neighbourhood of $\wt{x}$ which sends horizontal disks to $W^{\wt{x}}_{\infty}$-plaques. This proves that these plaques form a foliation. Since the leaves of the foliations are tangent to small cones around the $E^{cs}_g$ direction and these leaves are $\wt{g}$-invariant, we get that the foliation is tangent to $E^{cs}_g$. Finally observe that the foliation $\wt{\cW}^{cs}_g$ has the same properties that $\wt{\cF}^{cs}_g$. Thus we have global product structure between $\wt{\cW}^{cs}_g$ and $\wt{\cW}^{u}_g$.
\end{proof}

\begin{cor} \label{ccdi}
	If $g \in \textnormal{PH}_f(M)$ verifies the following conditions:
	\begin{itemize}
		\item $g$ is $u$ and $s$ proper.
		\item $g$ is \textnormal{SADC} with global product structure. 
			\end{itemize}
	Then $g$ is dynamically coherent, center fibered and has global product structure. 
\end{cor}

We end this section with a proposition which finishes the proof of the equivalence between dynamically coherence and center fibered, with $\sigma$ properness and SADC (in presence of global product structure).  

\begin{prop} \label{propequivfinal}
	If $g\in \textnormal{PH}_f(M)$ is dynamically coherent, center fibered and has global product structure, then it is $\sigma$-proper ($\sigma=s,u$) and \textnormal{SADC} with global product structure. 
	\end{prop}	

\begin{proof}
	Take a dynamically coherent and center fibered $g\in \textnormal{PH}_f(M)$, such that $\wt{\cW}^{cs}_g$ and $\wt{\cW}^{u}_g$ have global product structure, and $\wt{\cW}^{cu}_g$ and $\wt{\cW}^{s}_g$ have global product structure. Suppose that there is $\wt{y} \in \wt{\cW}^{u}_g(\wt{x})$ such that $H_g(\wt{y})=H_g(\wt{x})$. Then by center fibered this implies that $\wt{y} \in \wt{\cW}^c_g(\wt{x})\subset \wt{\cW}^{cs}_g(\wt{x})$. But then $\{\wt{x},\wt{y}\}\in \wt{\cW}^{u}_g(\wt{x})\cap \wt{\cW}^{cs}_g(\wt{x})$ which violates the global product structure. This implies that $H_g|_{\wt{\cW}^{u}_g}$ is injective, and therefore $g$ is $u$-proper by Lemma \ref{equipro}. The case $s$-proper is exactly the same.
	
	Now recall that $\cW^{cs}_g$ and $\cW^{cu}_g$ are uniformly transverse to $E^{u}_g$ and $E^{s}_g$ respectively, and so in order to prove that $g$ is SADC, it remains to show that $\wt{\cW}^{cs}_g$ and $\wt{\cW}^{cu}_g$ are almost parallel to the center-stable and center-unstable foliations of $f$. This is quite direct, since $\sigma$-properness, global product structure and center fibered implies that  $H_g(\wt{\cW}^{cs}_g(\wt{x}))=\cW^{s}_{\wt{f}_c}(H_g(\wt{x}))$ and $H_g(\wt{\cW}^{cu}_g(\wt{x}))=\cW^{u}_{\wt{f}_c}(H_g(\wt{x}))$ for every $\wt{x}\in \wt{M}$. This implies SADC because $H_g$ is at bounded distance from $\wt{p}=H_f$.  
\end{proof}

\section{Dynamical coherence is open and closed} \label{sectionopenandclosed}

To obtain the main theorem of this article, we have to prove that SADC, $\sigma$-properness ($\sigma=s,u$) and global product structure (between the strong stable/unstable manifolds and the ones given by SADC) are $C^1$ open and closed properties among partially hyperbolic diffeomorphisms in $M$ isotopic to $f$. Then we can apply Corollary \ref{ccdi} to a whole connected component as long as it contains a diffeomorphism with such properties.

\subsection{SADC is $C^{1}$ open.}

\begin{prop} \label{tsadca} \textnormal{SADC} is a $C^1$ open property among $\textnormal{PH}_f(M)$.
\end{prop} 
\begin{proof}
	This is pretty direct since the same foliation works by the continuity of the $E^{s}$ and $E^{u}$ bundles. Take $g\in \text{PH}_f(M)$ with SADC property and let $\mathcal{F}^{cs}_{g}$, $\mathcal{F}^{cu}_{g}$ be the foliations given by the SADC property. These foliations are transverse to $E^{u}_g$, $E^{s}_g$ and their lifts are almost parallel to $\wt{\cW}^{cs}_f$ and $\wt{\cW}^{cu}_f$ respectively. Then $\angle (\wt{\cF}^{cs}_g(\wt{x}),E^{u}_g(\wt{x}))>\e$ for every $\wt{x}\in \wt{M}$ and there is $\mathcal{U}(g)$ a neighbourhood of $g$ in the $C^{1}$ topology s.t. for every $g' \in \mathcal{U}(f)$ we have $\angle (E^{u}_{g'}(\wt{x}),E^{u}_{g}(\wt{x}))<\frac{\epsilon}{2}$ for every $\wt{x}\in \wt{M}$. Take $\mathcal{F}^{cs}_{g'}=\mathcal{F}^{cs}_{g}$, then
	\begin{eqnarray*}
		\angle(\wt{\cF}^{cs}_{g'}(\wt{x}),E^{u}_{g'}(\wt{x}))+\frac{\e}{2} & > & \angle (\wt{\cF}^{cs}_{g'}(\wt{x}),E^{u}_{g'}(\wt{x}))+\angle (E^{u}_{g'}(\wt{x}),E^{u}_{g}(\wt{x}))\\
		&\geq & \angle (\wt{\cF}^{cs}_{g'}(\wt{x}),E^{u}_{g}(\wt{x}))\\
		&=&\angle (\wt{\cF}^{cs}_g(\wt{x}),E^{u}_g(\wt{x}))>\e >0
	\end{eqnarray*}
	This implies that $\angle (\wt{\cF}^{cs}_{g'}(\wt{x}),E^{u}_{g'}(\wt{x}))>\frac{\e}{2}$ for every $\wt{x} \in \wt{M}$.
	Then every $g'\in \mathcal{U}(g)$ has foliations $\mathcal{F}^{cs}_{g'}$, $\mathcal{F}^{cu}_{g'}$ transverse to $E^{u}_{g'}$, $E^{s}_{g'}$ and thus each $g'\in \mathcal{U}(g)$ verifies SADC.
\end{proof}	

\subsection{$\sigma$-proper is $C^1$ open}

The following is a classical fact about hyperbolic theory that we will need for the next proposition.

\begin{rem} \label{rkhyp}
	Given $f\in \text{PH}(M)$, there exist constants $1<\lambda_{f}<\Delta_{f}$ and there exists $\mathcal{U}$ a $C^{1}$-neighbourhood of $f$ s.t. for every $g\in \mathcal{U}$, $\wt{x}\in \wt{M}$ and $R>0$ we have:
	$$\wt{\cW}^{u}_g(\wt{g}(\wt{x}),\lambda_{f}R)\subset \wt{g}(\wt{\cW}^{u}_g(\wt{x},R))\subset \wt{\cW}^{u}_g(\wt{g}(\wt{x}),\Delta_{f}R)
	$$ Analogously for $\wt{\cW}^{s}_g$ by applying $\wt{g}^{-1}$.
\end{rem}

\begin{prop} \label{pspopen}
	
	$\sigma$-proper is a $C^1$ open property among $\textnormal{PH}_f(M)$, for $\sigma=s,u$. 
	
	\begin{proof}
		Given $g\in \text{PH}_f(M)$ that is $\sigma$-proper, we must find a neighbourhood $\mathcal{U}(g)$ in the $C^{1}$ topology such that every $g'\in \mathcal{U}(g)$ is $\sigma$-proper. Remark \ref{osp1} says that it's enough to find a neighbourhood $\mathcal{U}(g)$ and $R>0$ such that for every  $g'\in \mathcal{U}(g)$, $\wt{x}\in \wt{M}$ :
		$$ (H_{g'})^{-1}(D^{\sigma}(H_{g'}(\wt{x}),1))\cap \wt{\cW}_{g'}^{\sigma}(\wt{x})\subseteq \wt{\cW}^{\sigma}_{g'}(\wt{x},R)
		$$
		Since $g$ is $\sigma$-proper, we know that $H_g|_{\wt{\cW}^{\sigma}_g(\wt{x})}:\wt{\cW}^{\sigma}_g(\wt{x})\to \cW^\sigma_{\wt{f}_c}(H_g(\wt{x}))$ is a homeomorphism. Then there is $R_{1}>0$ s.t. $$H_g(\wt{\cW}^{\sigma}_g(\wt{x},R_1)^c)\cap D^{\sigma}(H_g(\wt{x}),2)=\emptyset ;\forall\;\wt{x} \in \wt{M}. $$
		Call $A^{\sigma}_{r,R,g'}(\wt{x})$ the annulus $\wt{\cW}^{\sigma}_{g'}(\wt{x},R)\setminus \wt{\cW}^{\sigma}_{g'}(\wt{x},r)$ for $R>r>0$. Then for $R_{2}>\Delta_{g}R_{1}$ we have that
		$$H_g(A^{\sigma}_{R_1,R_2,g}(\wt{x}))\cap D^{\sigma}(H_g(\wt{x}),2)=\emptyset
		$$ where we take $\Delta_g>1$ like in Remark \ref{rkhyp}. Now since $H_g$ is continuous and $\Gamma$-invariant, it is uniformly continuous. Then there is $\e_1>0$ s.t. if $d(\wt{x},\wt{y})<\e_1$ then $d(H_g(\wt{x}),H_g(\wt{y}))<1/4$. Take the following $C^{1}$-neighbourhoods:
		\begin{itemize}
			\item From uniform hyperbolicity there is $\cU_1(g)$ such that the constants $\Delta_{g}$ and $\lambda_{g}$ are uniform in $\cU_1(g)$ (see Remark \ref{rkhyp}).
			\item The continuous variation of the leaves in the $C^{1}$ topology says that for every $\epsilon_{1}>0$ and $R_{2}>0$, there is $\cU_2(g)$ and $\delta>0$ s.t. for every $g'\in \cU_{2}(g)$ and every pair of points $\wt{x},\wt{y}$ with $d(\wt{x},\wt{y})<\delta$ we have $d_{C^{1}}(\wt{\cW}^{\sigma}_{g'}(\wt{x},R_{2}),\wt{\cW}^{\sigma}_{g'}(\wt{y},R_{2}))<\epsilon_{1}$.
			\item Take $\cU_{3}(g)=\{g'\in \textnormal{PH}_{f}(M): d_{C^{0}}(H_{g'},H_g)<1/4\}$.
		\end{itemize}
		Finally take $\cU_g:=\cU_1(g)\cap\cU_2(g)\cap \cU_3(g)$.
		Now, let $g'\in \cU(g)$ and $\wt{x},\wt{y}$ such that $\wt{y}\in A^{\sigma}_{R_1,R_2,g'}(\wt{x})$. Then there is $\wt{z}\in A^{\sigma}_{R_1,R_2,g}(\wt{x})$ such that $d(\wt{z},\wt{y})<\e_1$ and from uniform continuity we get $dist(H_g(\wt{z}),H_g(\wt{y}))<1/4$. Since $\wt{z}\in A^{\sigma}_{R_1,R_2,g}(\wt{x})$ and $dist(H_g(\wt{z}),H_g(\wt{y}))<1/4$, applying the triangular inequality we obtain:
		\begin{eqnarray*}
			2&<&dist(H_g(\wt{z}),H_g(\wt{x}))\leq dist(H_g(\wt{z}),H_g(\wt{y}))+dist(H_g(\wt{y}),H_g(\wt{x})) \\ 
			&\leq & 1/4+ dist(H_g(\wt{y}),H_g(\wt{x}))
		\end{eqnarray*}
		Therefore $dist(H_g(\wt{y}),H_g(\wt{x}))>2-1/4$. Once again the triangular inequality gives:
		\begin{eqnarray*}
			2-1/4 &<& dist(H_g(\wt{y}),H_g(\wt{x})) \\
			&\leq& dist(H_g(\wt{y}),H_{g'}(\wt{y}))+dist(H_{g'}(\wt{y}),H_{g'}(\wt{x}))+dist(H_{g'}(\wt{x}),H_g(\wt{x})) \\
			&\leq & 1/4 +dist(H_{g'}(\wt{y}),H_{g'}(\wt{x}))+1/4
		\end{eqnarray*} and we conclude that $dist(H_{g'}(\wt{y}),H_{g'}(\wt{x}))>2-3/4>1$, which means 
		\begin{equation} \label{ecanillo}
			H_{g'}(A^{\sigma}_{R_1,R_2,g'}(\wt{x}))\cap D^{\sigma}(H_{g'}(\wt{x}),1)=\emptyset \ \ \text{for every} \ \ \wt{x}\in \wt{M}
		\end{equation} 
		Finally this implies
		$$(H_{g'})^{-1}(D^{\sigma}(H_{g'}(\wt{x}),1))\cap \wt{\cW}_{g'}^{\sigma}(\wt{x})\subseteq \wt{\cW}_{g'}^{\sigma}(\wt{x},R_1) \ \ \text{for every} \ \ \wt{x}\in \wt{M}
		$$
		If it weren't the case, there will be $\wt{y}\in \wt{\cW}^{\sigma}_{g'}(\wt{x})$ such that $H_{g'}(\wt{y})\in D^{\sigma}(H_{g'}(\wt{x}),1)$ but $\wt{y}\notin \wt{\cW}^{\sigma}_{g'}(\wt{x},R_2)$. By the choice of $\Delta_g$ we know that there is $n\in \Z$ (positive if $\sigma=s$ or negative if $\sigma=u$) s.t. $\wt{g'}^n(\wt{y})\in A^{\sigma}_{R_1,R_2,g'}(\wt{g'}^n(\wt{x}))$ and $H_{g'}(\wt{g'}^n(\wt{y}))\in D^{\sigma}(H_{g'}(\wt{g'}^n(\wt{x})),1)$. This contradicts \eqref{ecanillo} above. 
	\end{proof}
	
\end{prop}

\subsection{SADC + $\sigma$-proper + GPS is $C^1$ open}

In this subsection we are going to prove that given $g\in \text{PH}_f(M)$ which is $\sigma$-proper and SADC with global product structure, then every $g'$ sufficiently $C^1$ close to $g$ is also $\sigma$-proper and SADC with global product structure (maybe with a different foliation than the original one). 

\begin{prop} \label{gpsopen}
	Let $g\in \textnormal{PH}_f(M)$ be such that $g$ is $\sigma$-proper, for $\sigma=s,u$, and SADC with global product structure. Then there is a $C^1$ neighbourhood $\cU$ of $g$ such that every $g'\in \cU$ is $\sigma$ proper and SADC with global product structure. 
	
\end{prop}

\begin{proof}
	Take $g\in \text{PH}_f(M)$ such that $g$ is $\sigma$-proper for $\sigma=s,u$, and SADC with their corresponding foliations $\wt{\cF}^{cs}_g$ and $\wt{\cF}^{cu}_g$, and suppose that $\wt{\cF}^{cs}_g$ and $\wt{\cW}^{u}_g$ have global product structure (the other case is symmetric). By Theorem \ref{tcdi} we now that $g$ is dynamically coherent, center-fibered and $\wt{\cW}^{cs}_g$ and $\wt{\cW}^{u}_g$ have global product structure. 
	
	Now we can replace $\wt{\cF}^{cs}_g$ by $\wt{\cW}^{cs}_g$ in the SADC definition of $g$ (i.e. with these new foliations $g$ is still SADC by Proposition \ref{propequivfinal}). We have to do this interchange because we need $\Gamma$-invariance of the foliations (this will be clear in a moment).  
	
	By Proposition \ref{tsadca} we know there is a $C^1$ neighbourhood $\cU_1$ of $g$ such that every $g'\in \cU_1$ is SADC (applying the proposition to $\wt{\cW}^{cs}_g$). 
	
	On the other hand by Proposition \ref{pspopen} there is a $C^1$ neighbourhood $\cU_2$ of $g$ such that every $g' \in \cU_2$ is $\sigma$-proper. Moreover we know there is $R_1>0$ such that:
	\begin{equation} \label{equation01}
		(H_{g'})^{-1}(D^{\sigma}(H_{g'}(\wt{x}),1))\cap \wt{\cW}_{g'}^{\sigma}(\wt{x})\subseteq \wt{\cW}_{g'}^{\sigma}(\wt{x},R_1)
	\end{equation}
	for every $\wt{x}\in \wt{M}$ and $g'\in \cU_2$.
	
	\begin{claim}
		There is a $C^1$ neighbourhood $\cU_3$ of $g$ such that for every $g'\in \cU_3$ and every $\wt{x}\in \wt{M}$ we have that: 
		\begin{equation} \label{equation02}
			\wt{\cW}^{u}_{g'}(\wt{x},R_1)\cap \wt{\cW}^{cs}_g(\wt{x})=\{\wt{x}\}
		\end{equation}
	\end{claim}
	\begin{proof} Just notice that for every $\wt{x}\in \wt{M}$ there is $\e(\wt{x})>0$ and a $C^1$ neighbourhood $\cU(\wt{x})$ of $g$ such that for every $g'\in \cU(\wt{x})$ and every $\wt{y}\in B(\wt{x},\e(\wt{x}))$ Equation \eqref{equation02} holds. Since $\wt{\cW}^{cs}_g$ is $\Gamma$ invariant, we can restrict ourselves to a compact fundamental domain. Then, we can cover this fundamental domain by finite balls $B(\wt{x}_1,\e(\wt{x_1})),\dots, B(\wt{x}_{N},\e(\wt{x}_N))$ and take $\cU_3=\cap_{j=1}^{N} \cU(\wt{x}_j)$. This proves the claim.
	\end{proof}
	To end the proof of the proposition, take $g'\in \cU:=\cU_1\cap \cU_2 \cap \cU_3$ and take two points $\wt{x},\wt{y}\in \wt{M}$. Now it is easy to see that $\wt{\cW}^{u}_{g'}(\wt{x})\cap \wt{\cW}^{cs}_g(\wt{y})$ is non empty. By Equation \eqref{equation01} and Equation \eqref{equation02} of the claim, we have that $\wt{\cW}^{u}_{g'}(\wt{x})\cap \wt{\cW}^{cs}_g(\wt{y})$ is exactly one point. This proves the global product structure between $\wt{\cW}^{u}_{g'}$ and $\wt{\cW}^{cs}_g$. 	
\end{proof}

\begin{rem}
	In the proof of the previous claim we need the foliation to be $\Gamma$-invariant, in order to restrict ourselves to points in a fundamental domain (to be able to take a finite cover). 
	That's why we interchange $\wt{\cF}^{cs}_g$ with $\wt{\cW}^{cs}_g$ in the proof. 
\end{rem}

\subsection{SADC + $\sigma$-proper + GPS is $C^1$ closed}

The previous proposition shows that $\sigma$-properness and SADC with global product structure are $C^1$ open among $\text{PH}_f(M)$. To finish the proof of the main theorem we have to prove that they are also $C^1$-closed properties. This is the most difficult part of the theorem. For the proof we are going to use once again Theorem \ref{tcdi}. We first show that SADC is closed among the ones having SADC and GPS. 

\begin{prop} \label{tsadcc}
	Let $g_n\in \text{PH}_f(M)$ such that $g_n \xrightarrow{C^1} g$ and every $g_n$ is SADC and GPS. Then $g$ satisfy SADC. \end{prop}	
\begin{proof} 
	 Call $E^{cs}_n=E^{s}_{g_n}\oplus E^c_{g_n}$ and let ${\cF}^{cs}_n$, ${\cF}^{cu}_n$ be the foliations given by the SADC property for every $n\in \N$. By the $C^{1}$ convergence we have $E^{cs}_n\rightarrow E^{cs}_g$ and $E^{u}_n\rightarrow E^{u}_g$. Let $\eta=\angle (E^{cs}_g,E^{u}_g)$ (minimum bound of the angle). Now since $E^{cs}_n\rightarrow E^{cs}_g$ there is $n_1>0$ such that $\angle (E^{u}_g,E^{cs}_{n_1})>\frac{\eta}{2}$. Take ${\cF}^{cs}_{n_1}$ foliation uniformly transverse to $E^{u}_{n_1}$. Then there is $n_2>0$ such that $g_{n_1}^{-n_2}({\cF}^{cs}_{n_1})$ is contained in a cone centered at $E^{cs}_{n_1}$ of radius $\frac{\eta}{2}$. Thus $g_{n_1}^{-n_2}({\cF}^{cs}_{n_1})$ is uniformly transverse to $E^{u}_g$. Therefore, it is enough to show that $\wt{g}_{n_1}^{-n_2}(\wt{\cF}^{cs}_{n_1})$ is almost parallel to $\wt{\cW}^{cs}_f$. 
	
	To finish the proof, it's suffices  to show that if $\wt{\cF}^{cs}_g$ is almost parallel to $\wt{\cW}^{cs}_f$ and have GPS then $\wt{g}^{-1}(\wt{\cF}^{cs}_g)$ is almost parallel to $\wt{\cW}^{cs}_f$ as well. Let $R>0$ from the definition of almost parallel and let $\wt{x}, \wt{y}$ be such that $d_H(\wt{\cF}^{cs}_g(\wt{x}), \wt{\cW}^{cs}_f(\wt{y}))<R.$ We have to find $R'$ just depending on $R$ and $g$ such that $d_H(\wt{g}^{-1}(\wt{\cF}^{cs}_g(\wt{x})), \wt{\cW}^{cs}_f(\wt{f}^{-1}(\wt{y}))<R.$ Observe that\begin{eqnarray*}	
	d_H(\wt{g}^{-1}(\wt{\cF}^{cs}_g(\wt{x})), \wt{\cW}^{cs}_f(\wt{f}^{-1}(\wt{y}))&\le &d_H(\wt{g}^{-1}(\wt{\cF}^{cs}_g(\wt{x})), \wt{g}^{-1}(\wt{\cW}^{cs}_f((\wt{y})))\\& &+\;\; d_H(\wt{g}^{-1}(\wt{\cW}^{cs}_f((\wt{y})), \wt{\cW}^{cs}_f(\wt{f}^{-1}(\wt{y})))
	\end{eqnarray*}
By Lemma \ref{l.prod} applied to $g$ (the same exact proof works), there exists $r>0$ such that 	
	$$\wt{\cW}^{cs}_f(\wt{y})\subset\bigcup_{\wt{z}\in \wt{\cF}_g^{cs}(\wt{x})}\wt{\cW}^{u}_g(\wt{z}, r)$$
		and so $d_H(\wt{g}^{-1}(\wt{\cF}^{cs}_g(\wt{x})), \wt{g}^{-1}(\wt{\cW}^{cs}_f(\wt{y})))<\lambda_g^{-1}r.$
	On the other hand, $\wt{g}$ is at bounded distance from $\wt{f}$, say $K,$   and then $d_H(\wt{g}^{-1}(\wt{\cW}^{cs}_f(\wt{y})), \wt{f}^{-1}(\wt{\cW}^{cs}_f(\wt{y})))<K.$

\end{proof}

Before getting into the proof of the next theorem, let us define for $\sigma=s,u$:
$$\Pi^\sigma_{[\wt{x}]}:\wt{M}_c\to \cW^\sigma_{\wt{f}_c}([\wt{x}])$$
the holonomy by the corresponding foliation onto the leaf $\cW^\sigma_{\wt{f}_c}([\wt{x}])$.

\begin{thm} \label{tsadcypayc}
	For $\sigma=s,u$, being $\sigma$-proper and SADC with global product structure is a $C^1$-closed property in $\textnormal{PH}_f(M)$.
\end{thm}

\begin{proof}
	Take a sequence $\{g_k\}\subset \text{PH}_f(M)$ with $g_k\to g$ in the $C^{1}$ topology, such that for every $k\in \N$, $g_k$ is $\sigma$-proper and SADC with global product structure. By Proposition \ref{tsadcc} we know that $g$ is SADC. We have to prove that $g$ is $\sigma$-proper and that we have global product structure. We are going to prove case $\sigma=u$, but the case $\sigma=s$ is completely symmetric.  
	
	Note that every $g_k$ is in the hypothesis of Theorem \ref{tcdi}, then for every $k\in \N$ there is a $g_k$-invariant foliation $\cW^{cs}_{g_k}$ tangent to $E^{s}_{g_k}\oplus E^c_{g_k}$ such that:
	\begin{equation} \label{equ1}
		\wt{\cW}^{cs}_{g_k}(\wt{x})=(H_{g_k})^{-1}(\cW^s_{\wt{f}_c}(H_{g_k}(\wt{x})))
			\end{equation} 
	Notice also that by center-fibered we have that: 
	\begin{equation} \label{ecscf}
		H_{g_k}(\wt{x})=H_{g_k}(\wt{y}) \ \ \textnormal{if and only if} \ \ \wt{y}\in \wt{\cW}^c_{g_k}(\wt{x})
	\end{equation}
	\begin{claim}
		Given $\e>0$, there exists $\delta>0$, a cone field $\cC^{u}$ around $E^{u}_g$ and  $k_0$ such that if $k \geq k_0$ and $\wt{D}$ is a disk tangent to $\cC^{u}$ of internal radius larger than $\e$ and centered at $\wt{x}$, then
		$$D^{u}(H_{g_k}(\wt{x}),\delta)\subset \Pi^{u}_{H_{g_k}(\wt{x})}\circ H_{g_k}(\wt{D})
		$$ 
	\end{claim}
	\begin{proof}
		This is because $g_k \to g$ in the $C^{1}$ topology, and so $E^{\sigma}_k\to E^{\sigma}_g$ for every $\sigma$. Then $M$ has a finite cover of local product structure boxes $B$ of size smaller than $\e$ such that for $k \geq k_0$ large enough, these are local product structure boxes for $g_k$ too. We can take these boxes $B$ small enough to have:
		\begin{itemize}
			\item The boxes $2B$ and $3B$ are also local product structure boxes for $g_k$.
			\item For every $B$ of the covering and every disk $D\subset M$ tangent to $\cC^{u}$ of internal radius larger than $\e$ and centered at a point $x\in B$ we have that $D$ intersects in a unique point in $3B$ every center-stable plaque of $\cW^{cs}_{g_k}$ which intersects $2B$.
		\end{itemize} 
		\vspace{-0.5cm}
		\begin{figure}[H] 
			\begin{center}
				\includegraphics [width=16cm]{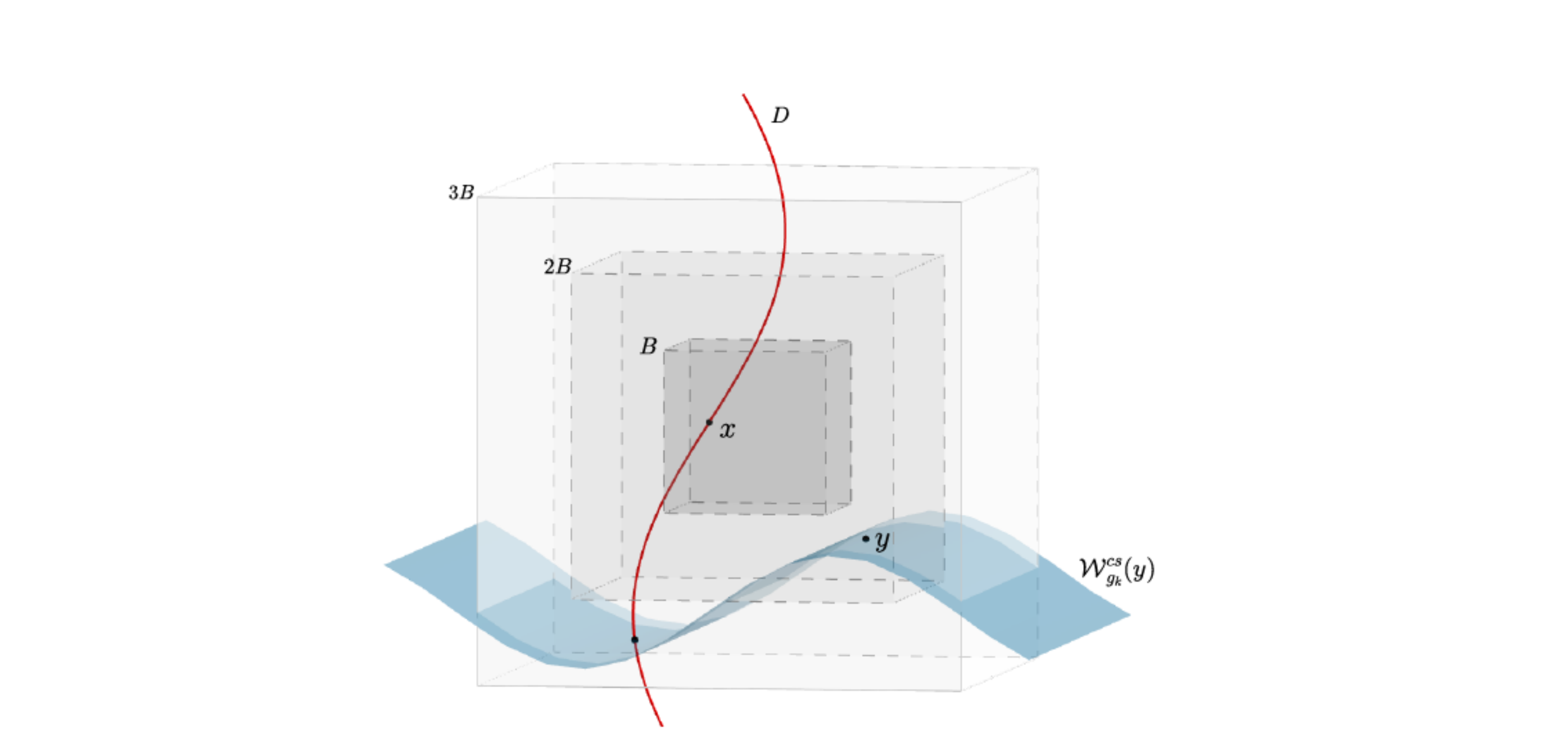}
				\caption{Boxes with local product structure.} 
			\end{center}
		\end{figure}	
		We can lift this cover by boxes and obtain a cover of $\wt{M}$ with the same properties as above. 
		The previous condition plus Equation \eqref{ecscf} implies that: $$\Pi^{u}_{H_{g_k}(\wt{x})}\circ H_{g_k}(2\wt{B})\subset \Pi^{u}_{H_{g_k}(\wt{x})}\circ H_{g_k}(\wt{D})$$ where $\wt{D}\subset \wt{M}$ is a lift of a disk $D\subset M$ as above. 
		Using the injectivity of $H_{g_k}$ restricted to $\wt{\cW}^{u}_{g_k}$ leaves, we have that given a connected component $2\wt{B}$ of a lift we have $int(\Pi^{u}_{H_{g_k}(\wt{x})}\circ H_{g_k}(2\wt{B}))\neq \emptyset
		$ and every point $\wt{y} \in \wt{B}$ verifies that $\Pi^u_{H_{g_k}(\wt{x})}(H_{g_k}(\wt{y}))$ lies in the interior of $\Pi^u_{H_{g_k}(\wt{x})}\circ H_{g_k}(2\wt{B})$. Since there are finite boxes (in $M$), there is a uniform $\delta>0$ such that $\Pi^u_{H_{g_k}(\wt{x})}(H_{g_k}(\wt{B}))$ is at bounded $\delta$ distance from the boundary of $\Pi^u_{H_{g_k}(\wt{x})}\circ H_{g_k}(2\wt{B})$ independently of the box $\wt{B}$.
		We deduce that every disk $\wt{D}$ of internal radius $\e$ and centered at $\wt{x}$ and tangent to a small cone around $E^{u}_g$ verifies that $\Pi^u_{H_{g_k}(\wt{x})}\circ H_{g_k}(\wt{D})$ contains $D^u(H_{g_k}(\wt{x}),\delta)$ as desired.  
	\end{proof}
	
	\begin{claim}
		For $k$ sufficiently large enough and for every pair of points $\wt{x},\wt{y} \in \wt{M}$, we have that $\wt{\cW}^{u}_g(\wt{x})$ and $\wt{\cW}^{cs}_{g_k}(\wt{y})$ have non-trivial intersection. 
	\end{claim}
	
	\begin{proof}
		Given two points $\wt{x},\wt{y} \in \wt{M}$, take an arc $\cS$ in $\cW^u_{\wt{f}_c}(H_{g_k}(\wt{x}))$ joining
		$H_{g_k}(\wt{x})$ and $\Pi^u_{H_{g_k}(\wt{x})}(H_{g_k}(\wt{y}))$. Fix $\e>0$ and take the corresponding $\delta>0$, the cones $\cC^{u}$ and $k_0>0$ from the previous claim. Then we have that 
		$$\Pi^u_{H_{g_k}(\wt{x})}\circ H_{g_k}(\wt{\cW}^{u}_g(\wt{x},\e)) \supset D^u(H_{g_k}(\wt{x}),\delta)$$ In the same way we get that
		$$\Pi^u_{H_{g_k}(\wt{x})}\circ H_{g_k}(\wt{\cW}^{u}_g(\wt{x},2\e)) \supset D^u(H_{g_k}(\wt{x}),2\delta)$$ We can apply inductively the same argument, and since the arc $\cS$ is compact, we get $m\in \N$ such that 
		$$\Pi^u_{H_{g_k}(\wt{x})}\circ H_{g_k}(\wt{\cW}^{u}_g(\wt{x},m\e)) \supset D^u(H_{g_k}(\wt{x}),m\delta) \supset \cS$$
		Then there is a point $\wt{z}\in \wt{\cW}^{u}_g(\wt{x},m\e) \subset \wt{\cW}^{u}_g(\wt{x})$ such that $$\Pi^u_{H_{g_k}(\wt{x})}\circ H_{g_k}(\wt{z})=\Pi^u_{H_{g_k}(\wt{x})}\circ H_{g_k}(\wt{y})$$
		Then $H_{g_k}(\wt{z})\in\cW^s_{\wt{f}_c}(H_{g_k}(\wt{y}))$ and this implies by Equation \eqref{equ1} that $\wt{z}\in \wt{\cW}^{cs}_{g_k}(\wt{y})$. We conclude that $\wt{z}\in \wt{\cW}^{u}_g(\wt{x})\cap \wt{\cW}^{cs}_{g_k}(\wt{y})$ as desire.
		
			\end{proof}
	
	\begin{claim} 
		For $k$ sufficiently large, the foliations $\wt{\cW}^{u}_g$ and $\wt{\cW}^{cs}_{g_k}$ have global product structure. In particular, the map $\Pi^u_{H_{g}(\wt{x})}\circ H_{g_k}|_{\wt{\cW}^{u}_g(\wt{x})}:\wt{\cW}^{u}_g(\wt{x}) \to \cW^u_{\wt{f}_c}(H_{g}(\wt{x}))$ is a homeomorphism. 
	\end{claim}  	
		\begin{proof}	
		By the previous claim, we only have to prove that the intersection between $\wt{\cW}^{u}_g(\wt{x})$ and $\wt{\cW}^{cs}_k(\wt{y})$ is unique for every pair of points $\wt{x},\wt{y} \in \wt{M}$. Since the leaf $\wt{\cW}^{u}_g(\wt{x})$ intersects transversely  $\wt{\cW}^{cs}_{g_k}(\wt{y})$ for every $\wt{x},\wt{y}$ and $H_{g_k}(\wt{\cW}^{cs}_{g_k}(\wt{y}))=\cW^s_{\wt{f}_c}(H_{g_k}(\wt{y}))$ we have that $H_{g_k}(\wt{\cW}^{u}_g(\wt{x}))$ is topologically transverse to the stable foliation $\cW^s_{\wt{f}_c}$. This implies that  
		$$\Pi^u_{H_{g}(\wt{x})}\circ H_{g_k}:\wt{\cW}^{u}_g(\wt{x})\to \cW^u_{\wt{f}_c}(H_{g}(\wt{x}))
		$$
		is locally injective, hence a covering and, since $\cW^u_{\wt{f}_c}(H_{g}(\wt{x}))$ is contractible, it must be injective. This proves that $\Pi^u_{H_{g}(\wt{x})}\circ H_{g_k}$ restricted to $\wt{\cW}^{u}_g(\wt{x})$ is a homeomorphism onto $\cW^u_{\wt{f}_c}(H_{g}(\wt{x}))$. 		
		
	\end{proof} 
	This claim proves that $g$ is SADC with global product structure. To finish the proof of the theorem we must prove there is $R>0$ such that:
	$$ (H_g)^{-1}(D^u(H_g(\wt{x}),1))\cap \wt{\cW}^{u}_g(\wt{x})\subset \wt{\cW}^{u}_g(\wt{x},R) \ , \ \ \forall \wt{x}\in \wt{M}
	$$
	Fix $\wt{x}\in \wt{M}$. We know that $d_{C^{0}}(H_{g_k},H_g)<K_*$ for some constant $K_*>0$. Let $C_*$ be the corresponding constant as in Item \ref{p2properties} of Proposition \ref{p.propertiesc}. The previous claim says that the restriction of $\Pi^u_{H_{g}(\wt{x})}\circ H_{g_k}$ to $\wt{\cW}^{u}_g(\wt{x})$ is a homeomorphism onto $\cW^u_{\wt{f}_c}(H_{g}(\wt{x}))$. Then, we know there is $R_1=R_1(x)>0$ such that 
		$$\Pi^u_{H_{g}(\wt{x})}\circ H_{g_k}(\wt{\cW}^{u}_g(\wt{x},R_1)^c)\cap D^u(\Pi^u_{H_g(\wt{x})}(H_{g_k}(\wt{x})),1+2C_*)=\emptyset
	$$
	Take $\wt{y} \in \wt{\cW}^{u}_g(\wt{x},R_1)^c$. Then applying the triangular inequality  we obtain  
	\begin{eqnarray*}
		1+2C_* & < & dist(\Pi^u_{H_{g}(\wt{x})}( H_{g_k}(\wt{x})),\Pi^u_{H_{g}(\wt{x})}(H_{g_k}(\wt{y}))) \\
		& \leq & dist(\Pi^u_{H_{g}(\wt{x})}(H_{g_k}(\wt{x})),\Pi^u_{H_{g}(\wt{x})}(H_g(\wt{x}))) \\
		&+&dist(\Pi^u_{H_{g}(\wt{x})}(H_g(\wt{x})),\Pi^u_{H_{g}(\wt{x})}(H_g(\wt{y}))) \\
		&+&dist(\Pi^u_{H_{g}(\wt{x})}(H_g(\wt{y})),\Pi^u_{H_{g}(\wt{x})}(H_{g_k}(\wt{y})))\\
		& < &  C_* +dist(H_g(\wt{x}),H_g(\wt{y}))+ C_*
	\end{eqnarray*}
\begin{figure}[H]
	\begin{center}
		\includegraphics [width=13cm]{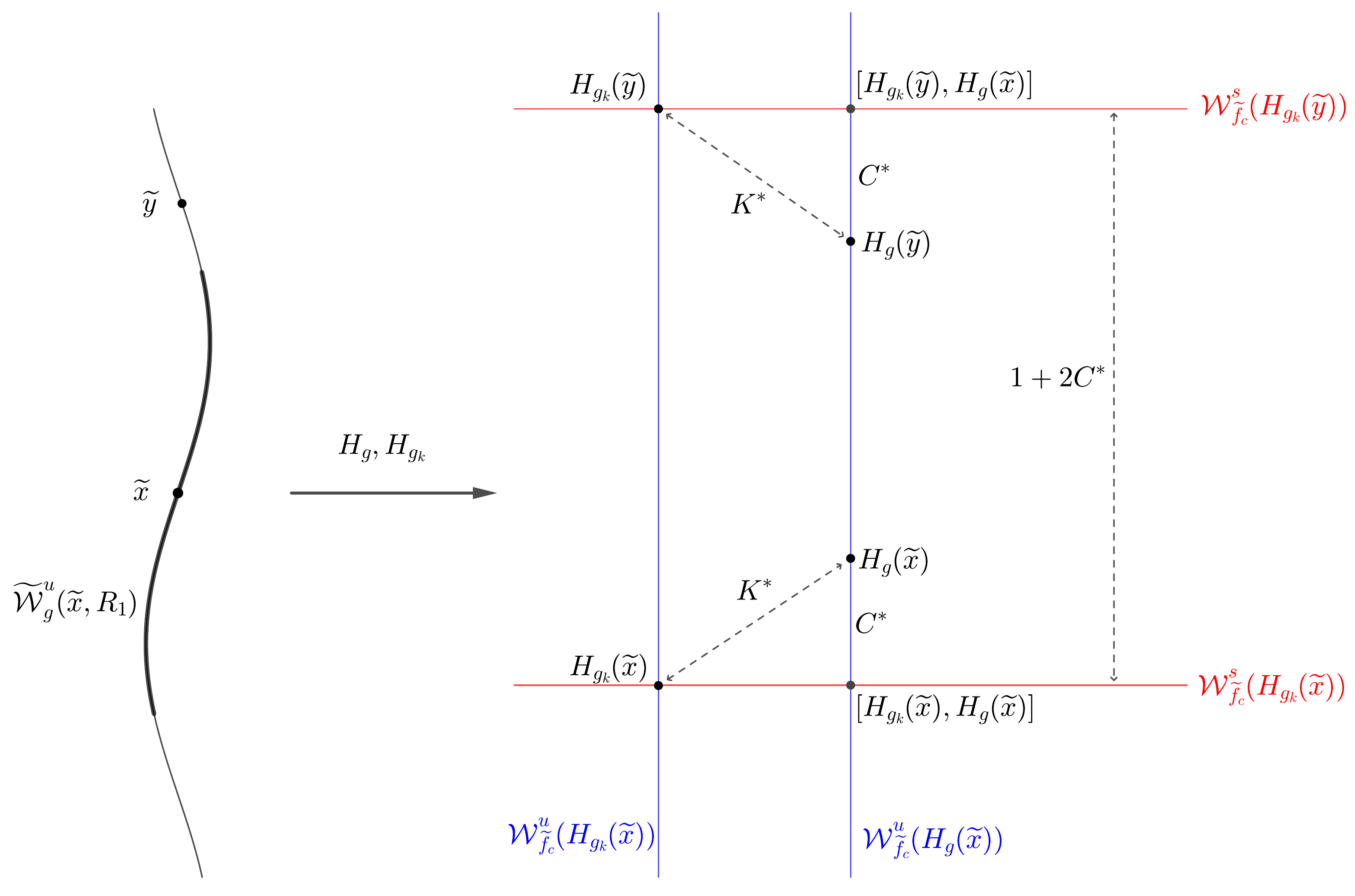}
		\caption{Triangular inequality} \label{figt}
	\end{center}
\end{figure}
where we have used that 
	\begin{itemize}
	
	 \item $\Pi^u_{H_{g}(\wt{x})}(H_g(\wt{x}))=H_g(\wt{x})$
	 \item $\Pi^u_{H_{g}(\wt{x})}(H_g(\wt{y}))=H_g(\wt{y})$ since $H_g(\wt{\cW}^{u}_g(\wt{x}))\subset \cW^u_{\wt{f}_c}(H_g(\wt{x}))$ 
	 \item  $\Pi^u_{H_g(\wt{x})}(H_{g_k}(\wt{x}))=[H_{g_k}(\wt{x}),H_g (\wt{x})]$ 
	\item $\Pi^u_{H_g(\wt{x})}(H_{g_k}(\wt{y}))=[H_{g_k}(\wt{y}),H_g (\wt{y})]$
	\end{itemize}
Thus $dist(H_g(\wt{x}),H_g(\wt{y})) >1$ (see Figure \ref{figt} above) and therefore we get 
	$$H_g(\wt{\cW}^{u}_g(\wt{x},R_1)^c)\cap D^u(H_{g}(\wt{x}),1)=\emptyset
	$$
	which is the same that
	$$(H_g)^{-1}(D^u(H_g(\wt{x}),1))\cap \wt{\cW}^{u}_g(\wt{x})\subset \wt{\cW}^{u}_g(\wt{x},R_1)
	$$
	Then we have proved that the function $\varphi$ is well defined where $$\varphi(x)=\inf \{ R>0: (H_g)^{-1}(D^u(H_g(\wt{x}),1))\cap \wt{\cW}^{u}_g(\wt{x})\subset \wt{\cW}^{u}_g(\wt{x},R)\}$$ By Remark \ref{osp1} we have to prove that $\varphi$ is uniformly bounded in $\wt{M}$ for getting $u$-proper. Since $\varphi$ is $\Gamma$-periodic (because $H_g$ is $\Gamma$-periodic), it's enough to restrict ourselves to points in a fundamental domain which is compact. Thus it is enough to show that if $\wt{x_n}\to \wt{x}$ then $\varphi(\wt{x_n})\leq \varphi (\wt{x})$. 
	To prove this, note that $H_g(\wt{\cW}^{u}_g(\wt{x},\varphi(\wt{x})))$ contains $D^u(H_g(\wt{x}),1)$. Now for every $\e>0$ we can find $\delta>0$ such that
	$$D^u(H_g(\wt{x}),1+\delta)\subset H_g(\wt{\cW}^{u}_g(\wt{x},\varphi(\wt{x})+\e))
	$$
	By continuous variation of the $\wt{\cW}^{u}_g$-leaves and since $H_g$ is continuous, we deduce that for $n$ large enough $H_g(\wt{\cW}^{u}_g(\wt{x_n},\varphi(\wt{x})+\e))$ contains $D^u(H_g(\wt{x_n}),1)$. This shows that $\limsup \varphi(\wt{x_n})\leq \varphi(\wt{x})+\e$. Since the choice of $\e>0$ was arbitrary, we get the desire result. 
\end{proof}

\subsection{Proof of the Theorem \ref{teoA} (Theorem \ref{a})} \label{subsectionproof}

Let $g\in \text{PH}_f(M)$ be a diffeomorphism in the same connected component of a partially hyperbolic diffeomorphism $g'$ such that:
\begin{itemize}
	\item $g'$ is dynamically coherent.
	\item $g'$ is center fibered.
	\item $\wt{\cW}^{cs}_{g'}$ and $\wt{\cW}^{u}_{g'}$ have GPS and, $\wt{\cW}^{cu}_{g'}$ and $\wt{\cW}^{s}_{g'}$ have GPS.
\end{itemize}
Then by Proposition \ref{propequivfinal} we have that $g'$ is $\sigma$ proper and SADC (and has GPS).  

Propositions \ref{tsadca}, \ref{pspopen}, \ref{gpsopen}, \ref{tsadcc} and Theorem \ref{tsadcypayc} tell us that $\sigma$-proper, SADC and global product structure are open and closed properties in the $C^1$ topology among $\text{PH}_f(M)$. In particular this implies that $g$ is $\sigma$-proper, SADC and has global product structure. By Theorem \ref{tcdi} (and Corollary \ref{ccdi}) we get that $g$ is dynamically coherent, center fibered and has global product structure. This ends the proof.

\section{Leaf conjugacy and proof of Theorem \ref{teoB} (Theorem \ref{b})} \label{sectionleafc}

In this section we are going to prove Theorem \ref{teoB} (Theorem \ref{b} in the introduction). For the proof we're going to show that center-fibered implies plaque expansiveness. Then we can conclude by Theorem \ref{thmplaqueexpansivetostablydc0} and a connectedness argument.

\begin{prop} \label{fispe}
	Every $g\in \textnormal{PH}_f^0(M)$ is plaque expansive.
	\begin{proof}
		Take $g\in \text{PH}_f^0(M)$. We know from Theorem \ref{teoA} that $g$ is dynamically coherent and center fibered. 
		Now take $\e>0$ and two $\e$-pseudo orbits $\{x_n\}_{n\in \Z}$ and $\{y_n\}_{n\in \Z}$ such that:
		\begin{itemize}
			\item[(i)] $g(x_n)$ belongs to the plaque $\cW^c_g(x_{n+1}, r)$, for every $n\in \Z$. 
			\item[(ii)] $g(y_n)$ belongs to the plaque  $\cW^c_g(y_{n+1},r)$, for every $n\in \Z$.
			\item[(iii)] $d(x_n,y_n)<\e$, for every $n\in \Z$.
		\end{itemize} Then we have to prove that $x_0$ and $y_0$ belong to the same center leaf. To do so, first take two lifts $\wt{x_0}$ and $\wt{y_0}$ of $x_0$ and $y_0$ respectively such that $d(\wt{x_0},\wt{y_0})<\e$. Since $\e$ is small enough, we have a unique pair of sequences $\{\wt{x_n}\}_{n\in \Z}$ and $\{\wt{y_n}\}_{n\in \Z}$ that check Items (i),(ii) and (iii).   
		
		Notice that center fibered imply that $H_g(\wt{g}(\wt{x_n}))=H_g(\wt{x_{n+1}})$ and $H_g(\wt{g}(\wt{y_n}))=H_g(\wt{y_{n+1}})$. By semiconjugacy we get
		$$\wt{f}_c(H_g(\wt{x_n}))=H_g(\wt{g}(\wt{x_n}))=H_g(\wt{x_{n+1}})$$ 
		$$\wt{f}_c(H_g(\wt{y_n}))=H_g(\wt{g}(\wt{y_n}))=H_g(\wt{y_{n+1}})$$
		Then $\{H_g(\wt{x_n})\}_{n\in \Z}$ and $\{H_g(\wt{y_n})\}_{n\in \Z}$ are orbits of the map $\wt{f}_c:\wt{M}_c\to \wt{M}_c$. These two orbits remains at a bounded distance and $\wt{f}_c$ is infinitely expansive.
		
				Therefore $H_g(\wt{x_0})=H_g(\wt{y_0})$. By center-fibered we conclude that $\wt{y_0}\in \wt{\cW}^c_g(\wt{x_0})$ and since its distance is less that $\e$  we have $y_0\in \cW^c_g(x_0,r)$ proving that $g$ is plaque-expansive.	
	\end{proof}	
\end{prop}

\begin{proof}[Proof of Theorem \ref{teoB}]
	Take $g_0$ and $g_1$ diffeomorphisms in the same connected component of $\text{PH}_f^0(M)$. Then there is a continuous path $\{g_t\}_{t\in [0,1]} \subset \text{PH}_f^0(M)$ connecting $g_0$ and $g_1$. 
	
	By Theorem \ref{teoA} every $g_{t}$ is dynamically coherent and center fibered. Then by Proposition \ref{fispe} every $g_t$ is plaque expansive. We can apply Theorem \ref{thmplaqueexpansivetostablydc0} (Theorem 7.1 in \cite{HPS}) to every $g_t$ and obtain a neighbourhood $\cU(t)$ such that every partially hyperbolic in $\cU(t)$ is leaf conjugate to $g_t$. Since $[0,1]$ is compact and connected, we can cover $\{g_t\}_{t\in [0,1]}$ by a finite union $\cup_{i=1}^{i=l}\cU(t_i)$. Since leaf-conjugacy is an equivalence relation we conclude that $g_0$ is leaf conjugate to $g_1$. 
\end{proof}

\small
	\vspace{0.5cm}
\begin{flushleft}
	\textsc{Luis Pedro Pi\~neyr\'ua}\\
	IMERL, Facultad de Ingenier\'ia\\
	Universidad de la Rep\'ublica, Montevideo, Uruguay\\
	email: \texttt{lpineyrua@fing.edu.uy} 
	
	\vspace{0.4cm}
	
	\textsc{Mart\'in Sambarino}\\
	CMAT, Facultad de Ciencias\\
	Universidad de la Rep\'ublica, Montevideo, Uruguay\\
	email: \texttt{samba@cmat.edu.uy} 
\end{flushleft}

\end{document}